\documentclass[11pt]{article}
\usepackage{amsfonts,amsmath,amssymb,amsthm}
\usepackage[mathscr]{euscript}
\usepackage{mathrsfs}
\usepackage{indentfirst}
\usepackage{color}
\usepackage{enumerate}
\usepackage{lineno}

\newtheorem{theorem}{Theorem}[section]
\newtheorem{lemma}[theorem]{Lemma}

\newtheorem{remark}[theorem]{Remark}
\newtheorem{proposition}[theorem]{Proposition}
\numberwithin{equation}{section}

\newcommand{\lbl}[1]{\label{#1}}
\allowdisplaybreaks

\newcommand{\be}{\begin{equation}}
\newcommand{\ee}{\end{equation}}
\newcommand\bes{\begin{eqnarray}} \newcommand\ees{\end{eqnarray}}
\newcommand{\bess}{\begin{eqnarray*}}
\newcommand{\eess}{\end{eqnarray*}}
\newcommand{\bbbb}{\left\{\begin{aligned}}
\newcommand{\nnnn}{\end{aligned}\right.}
\newcommand{\bea}{\begin{align*}}
\newcommand{\eea}{\end{align*}}

\newcommand\ep{\varepsilon}

\newcommand\kk{\left}
\newcommand\rr{\right}
\newcommand\dd{\displaystyle}

\newcommand\dx{{\rm d}x}
\newcommand\dy{{\rm d}y}

\newcommand\lm{\lambda}
\newcommand\nm{\nonumber}

\newcommand\yy{\infty}
\newcommand\qq{\eqref}
\newcommand\ol{\overline}
\newcommand\ud{\underline}
\newcommand\oo{\Omega}
\newcommand\boo{\ol\Omega}

\begin{document}\thispagestyle{empty}

\setlength{\abovedisplayskip}{7pt}
\setlength{\belowdisplayskip}{7pt}

\begin{center}
 {\LARGE\bf Longtime behaviors of a reducible cooperative system}\\[2mm]
 {\LARGE\bf with nonlocal diffusions and free boundaries}\footnote{This work was supported by NSFC Grants
 12301247, 12171120. This paper is not currently submitted to other journals and will not be submitted to other journals during the reviewing process.}\\[4mm]
{\Large Lei Li}\\[1mm]
{\small School of Mathematics and Statistics, Henan University of Technology, Zhengzhou, 450001, China}\\[3mm]

{\Large Mingxin Wang}\footnote{Corresponding author. {\sl E-mail}:  mxwang@hpu.edu.cn}\\[1mm]
{\small School of Mathematics and Information Science, Henan Polytechnic University, Jiaozuo, 454003, China}\end{center}
\date{\today}

\begin{abstract}
This paper aims at understanding the longtime behaviors of a reducible cooperative system with nonlocal diffusions and different free boundaries, describing the interactions of two mutually beneficial species. Compared with the irreducible and monostable cooperative system, the system we care about here has many nonnegative steady states, leading to much different and rich longtime behaviors. Moreover, since the possible nonnegative steady states on half space are non-constant, we need to employ more detailed analysis to understand the corresponding steady state problems which in turn helps us to derive a complete classification for the longtime behaviors of our problem. The spreading speeds of free boundaries and more accurate limits of $(u,v)$ as $t\to\infty$ are also discussed, and accelerated spreading can happen if some threshold conditions are violated by kernel functions.


\end{abstract}

\pagestyle{myheadings}
\section{Introduction}
{\setlength\arraycolsep{2pt}
\markboth{\rm$~$ \hfill A reducible cooperative system with nonlocal diffusions\hfill $~$}{\rm$~$ \hfill L. Li \& M.X. Wang\hfill $~$}
To describe the propagation of an invasive species which adopts a nonlocal diffusion strategy, Cao et al \cite{CDLL} proposed the following free boundary problem with nonlocal diffusion
\bes\label{1.5}\left\{\begin{aligned}
&u_t=d\int_{-\yy}^{\yy}J(x-y)u(t,y)\dy-du+f(u), & &t>0,~g(t)<x<h(t),\\
&u(t,x)=0,& &t>0, ~ x\notin(g(t),h(t)),\\
&h'(t)=\mu\int_{g(t)}^{h(t)}\!\!\int_{h(t)}^{\infty}
J(x-y)u(t,x)\dy\dx,& &t>0,\\
&g'(t)=-\mu\int_{g(t)}^{h(t)}\!\!\int_{-\infty}^{g(t)}
J(x-y)u(t,x)\dy\dx,& &t>0,\\
&h(0)=-g(0)=h_0>0,\;\; u(0,x)=u_0(x),& &x\in[-h_0,h_0],
 \end{aligned}\right.
 \ees
where $J$ satisfies
  \begin{enumerate}
\item[{\bf(J)}]$J\in C(\mathbb{R})\cap L^{\yy}(\mathbb{R})$, $J\ge 0$, $J(0)>0,~\dd\int_{\mathbb{R}}J(x)\dx=1$, \ $J$\; is even.
 \end{enumerate}
The reaction term $f(u)$ is of the Fisher-KPP type, i.e., $f\in C^1$, $f(0)=0<f'(0)$, $f(u)/u$ is strictly decreasing in $u>0$ and $f(u_*)=0$ for some $u_*>0$. The authors showed \eqref{1.5} is well-posed, and its longtime behaviors are governed by a spreading-vanishing dichotomy. When spreading happens, the spreading speed was studied by Du et al. \cite{DLZ}. They found that the spreading speed may be infinite, which is quite different from the local diffusion case (\cite{DL2010}). More precisely, it was proved in \cite{DLZ} that \eqref{1.5} has a unique finite spreading speed if and only if the kernel function $J$ satisfies
 \begin{enumerate}
   \item[{\bf(J1)}] $\dd\int_{0}^{\yy}xJ(x)\dx<\yy$.
 \end{enumerate}
The case where the spreading speed is infinite is usually called accelerated spreading which has been systematically investigated by Du and his collaborators \cite{DN1, DN2, DN3}. Among other things, the exact rate of accelerated spreading was recently derived by Du and Ni \cite{DN2}, in which a new method for constructing the suitable upper and lower solutions was put forward. For more development on the investigation of free boundary problem with nonlocal diffusion from ecology and epidemiology, for example, please refer to \cite{DN20}-\cite{LW24} or the expository article \cite{Du22}.

Motivated by the above work and to investigate the interactions of two mutually beneficial species, we in this paper investigate the following cooperative system with nonlocal diffusions and different free boundaries
\bes\left\{\begin{aligned}
&u_t=d_1\int_{0}^{\yy}J_1(x-y)u(t,y)\dy-d_1u+f_1(u,v), & &t>0,~x\in[0,s_1(t)),\\
&v_t=d_2\int_{0}^{\yy}J_2(x-y)v(t,y)\dy-d_2v+f_2(u,v), & &t>0,~x\in[0,s_2(t)),\\
&u(t,x)=0,~ x\in[s_1(t),\yy); ~ ~ ~ v(t,x)=0, ~ x\in[s_2(t),\yy), & &t>0,\\
&s'_1(t)=\mu_1\int_{0}^{s_1(t)}\!\!\int_{s_1(t)}^{\infty}
J_1(x-y)u(t,x)\dy\dx,& &t>0,\\
&s'_2(t)=\mu_2\int_{0}^{s_2(t)}\!\!\int_{s_2(t)}^{\infty}
J_2(x-y)v(t,x)\dy\dx,& &t>0,\\
&s_i(0)=s_{i0}>0,\\
&u=u_0(x),\;x\in[0,s_{10}];\;\;\;v=v_0(x),~ x\in[0,s_{20}], & & t=0,
 \end{aligned}\right.
 \label{1.1}\ees
where
 \[f_1(u,v)=r_1u\kk(a-u-\frac{u}{1+bv}\rr),\;\;\;f_2(u,v)=r_2v\kk(1-v-\frac{v}{1+qu}\rr),\]
and all parameters are positive. The initial functions $u_0$ and $v_0$ satisfy that $u_0\in C([0,s_{10}])$, $u_0(s_{10})=0<u_0(x)$ in $[0,s_{10})$, $v_0\in C([0,s_{20}])$ and $v_0(s_{20})=0<v_0(x)$ in $[0,s_{20})$. The kernel functions $J_i$ for $i=1,2$ satisfy the condition {\bf(J)}.

This model can be used to describe such a scenario in which the spatial movements of two species $u$ and $v$ are approximated by the integral operators, also known as nonlocal diffusion operators (see \cite{AMRT} for the biological interpretation), both $u$ and $v$ can across the fixed boundary $x=0$ but they will die immediately once they do it (corresponding to the homogeneous Dirichlet boundary conditions, see \cite{LLW22} for the explanation), and there is a cooperative relationship between $u$ and $v$.

There are three main differences between the model \qq{1.1} and free boundary problems of nonlocal diffusive cooperation systems studied in literature \cite{DN20, ZLW20, ZZLW20, DuNi22, CDu22, NV, WDu22} for epidemiological model (oral-fecal transmitted epidemic) and \cite{PLL23} for West Nile virus model. The first one is that the positive steady states of \qq{1.1} are non-constant, while those of the models studied in the above cited works is constant. The second one is that the model \qq{1.1} is reducible at $(0,0)$, but the models in \cite{DN20, ZLW20, ZZLW20, DuNi22, CDu22, NV, WDu22, PLL23} are irreducible at $(0,0)$. The last one is that the species $u$ and $v$ of \eqref{1.1} have their own habitats $[0,s_1(t))$ and $[0,s_2(t))$ that allows the existence  of a spatial area where only one species survives, while the species in the above cited works share the same habitat. In this paper, one will see that these differences bring about quite different dynamics.

The version of \eqref{1.1} with local diffusions, advection term and different free boundaries has been studied in \cite{LL, ZW, CT}. The well-posedness, longtime behaviors, criteria for spreading and vanishing, and estimates for spreading speed were obtained there. It is easy to see that in these existing literature, the possible positive steady states are constant, which eliminates the need for ones to investigate the corresponding steady state problems and also makes the iteration method effective for studying the longtime behaviors. However, for \eqref{1.1}, since the positive steady states are non-constant, the usual iteration method is invalid and we also need to utilize more subtle analysis to learn as many as possible about the steady state problems.

This paper is arranged as follows. Section 2 involves some preparatory results, mainly focused on the corresponding steady state problems on bounded interval and half space, respectively. Section 3 is devoted to the longtime behaviors of \eqref{1.1}. When spreading happens, the more accurate estimates for solution component $(u,v)$ are obtained in Section 4, in which we give a useful technical lemma and modify the existing iteration method. Finally, a discussion including some unsolved and interesting problems are given.

\section{Some preliminaries}
{\setlength\arraycolsep{2pt}

In this section, we give some preparatory results, including a principal eigenvalue problem, a free boundary problem, some evolving problems defined on bounded domain and some steady state problems. These results will pave the road for our later discussions.

Let $\oo=[l_1,l_2]$ with $l_1<l_2\in\mathbb{R}$, $P$ satisfy {\bf (J)} and constants $d, \alpha$ be positive. It is well-known (see e.g. \cite[Proposition 3.4]{CDLL}) that eigenvalue problem
 \bes\label{2.1}
d\int_\oo P(x-y)\phi(y)\dy-d\phi+\alpha\phi=\lambda\phi, ~ ~ x\in\oo,
\ees
has a unique principal eigenvalue $\lambda(d,P,\alpha,\oo)$. Moreover, $\lambda(d,P,\alpha,\oo)$ depends only on the length of interval $\oo$, $\lambda(d,P,\alpha,\oo)$ is continuous in $\oo$ and strictly increasing in $|\oo|$, and
 \bes
 \lim_{|\oo|\to0}\lambda(d,P,\alpha,\oo)=\alpha-d,\;\;\;\mbox{ and }\;\; \lim_{|\oo|\to\yy}\lambda(d,P,\alpha,\oo)=\alpha.
 \lbl{x2.2}\ees

\begin{proposition}[{\cite[Theorem 1.2]{DLZ}}]\label{p2.2}Let $P$ satisfy {\bf (J)} and constant $\beta$ be positive. Then semi-wave problem
\begin{eqnarray}\label{2.2}\left\{\begin{array}{lll}
 d\displaystyle\int_{-\infty}^{0}P(x-y)\phi(y){\rm d}y-d\phi+c\phi'+\phi(\alpha-\beta\phi)=0,\;\; \quad  x\in(-\yy,0),\\[3mm]
\phi(-\infty)=\dd\alpha/\beta,\ \ \phi(0)=0, \ \ c=\mu\int_{-\infty}^{0}\int_0^{\infty}\!P(x-y)\phi(x){\rm d}y{\rm d}x
 \end{array}\right.
 \end{eqnarray}
 has a unique solution pair $(c,\phi_c)$ with
  \bes
  c=c(d,P,\alpha,\beta,\mu)>0
  \lbl{x2.4}\ees
and $\phi_c'<0$ for $x\le0$ if and only if $P$ satisfies {\bf (J1)}.
\end{proposition}

Moreover, we recall some known results about the traveling wave solution of the Fisher-KPP equation with nonlocal diffusion. The following ``thin tailed'' condition is crucial.
\begin{enumerate}
   \item[{\bf(J2)}] $\dd\int_{0}^{\yy}e^{\lambda x}P(x)\dx<\yy$ for some $\lambda>0$.
 \end{enumerate}

\begin{proposition}[{\cite[Proposition 1.1]{CC}}, {\cite[Theorems 1 and 2]{Ya}} and {\cite[Theorems 5.1 and 5.2]{DLZ}}]\label{p2.3}Let $P$ satisfy {\bf (J)} and constant $\beta$ be positive. Then the following statements are valid.
\begin{enumerate}[$(1)$]
  \item If $P$ satisfies {\bf (J2)}, then there exists a constant $C_*>0$ such that problem
  \begin{eqnarray}\label{2.3}\left\{\begin{array}{lll}
 d\displaystyle\int_{-\infty}^{\yy}P(x-y)\phi(y){\rm d}y-d\phi+c\phi'+\phi(\alpha-\beta\phi)=0,\;\;\; x\in\mathbb{R},\\[3mm]
\phi(-\infty)=\alpha/\beta,\ \ \phi(\yy)=0
 \end{array}\right.
 \end{eqnarray}
 has a traveling wave solution if and only if $c\ge C_*$. The constant $C_*$ is usually called the minimal speed of \eqref{2.3} and can be formulated by
   \[C_*=\min_{\lambda>0}\frac{1}{\lambda}\left(d\int_{-\yy}^{\yy}P(x)e^{\lambda x}\dx-d+\alpha\right).\]
 \item If $P$ satisfies {\bf (J2)}, then the speed $c$ of semi-wave problem \eqref{2.2} will converge to $C_*$ as $\mu\to\yy$.
 \item If $P$ violates {\bf (J2)} but satisfies {\bf (J1)}, then the speed $c$ of semi-wave problem \eqref{2.2} will converge to $\yy$ as $\mu\to\yy$.
\end{enumerate}
\end{proposition}

Next we introduce a free boundary problem that has been studied in \cite{LLW22,LW24}. Let $(w,h)$ be the unique solution of problem
\bes\label{2.4}
\left\{\begin{aligned}
 & w_t=d\dd\int_{0}^{\yy}P(x-y)w(t,y)\dy-dw+w(\alpha-\beta w), && t>0,~ x\in[0,h(t)),\\
&w(t,x)=0, && t>0,~ x\ge h(t),\\
&h'(t)=\mu\dd\int_{0}^{h(t)}\!\!\int_{h(t)}^{\infty}
P(x-y)w(t,x)\dy\dx, && t>0,\\
&h(0)=h_0>0,\;\;  w(0,x)=w_0(x), &&x\in[0,h_0),
\end{aligned}\right.
 \ees
 where $w_0\in C([0,h_0])$, $w_0(x)$ is positive for $x\in[0,h_0)$ and vanishes at $x=h_0$.

 \begin{proposition}[{\cite[Theorems 3.2 and 3.4]{LLW22}} and {\cite[Theorem 1.1]{LW24}}]\label{p2.4}Let $(w,h)$ be the unique solution of \eqref{2.4}. Then one of the following alternatives must happen:

{\rm (i) {\bf Spreading:}} $\lim_{t\to\yy}h(t)=\yy$, and $\lim_{t\to\yy}w(t,x)=\theta(x)$ locally uniformly in $[0,\yy)$, where $\theta\in C([0,\yy))$ is the unique bounded positive solution of
 \bes\label{2.5}
d\int_{0}^{\yy}P(x-y)\theta(y)\dy-d\theta+\theta(\alpha-\beta\theta)=0, ~ ~x\in[0,\yy),
 \ees
and $\theta$ is strictly increasing in $[0,\yy)$, $0<\theta<\alpha/\beta$ and $\lim_{x\to\yy}\theta(x)=\alpha/\beta$. Moreover, if $P$ satisfies {\bf (J1)}, then
\bess
\lim_{t\to\yy}\frac{h(t)}{t}=c(d,P,\alpha,\beta,\mu) ~ ~ {\rm and ~ ~ }\lim_{t\to\yy}\max_{x\in[0,\,ct]}|w(t,x)-\theta(x)|=0,
\eess
where $c\in[0,c(d,P,\alpha,\beta,\mu))$ and $c(d,P,\alpha,\beta,\mu)$ is given by \qq{x2.4}. If $P$ violates {\bf (J1)}, then
\bess
\lim_{t\to\yy}\frac{h(t)}{t}=\yy ~ ~ {\rm and ~ ~ }\lim_{t\to\yy}\max_{x\in[0,\,ct]}|w(t,x)-\theta(x)|=0, ~ ~ \forall c\ge0.
\eess

{\rm (ii) {\bf Vanishing:}} $\lim_{t\to\yy}h(t)<\yy$, and $\lim_{t\to\yy}w(t,x)\to0$ uniformly in $[0,\yy)$.

Moreover, if $\lambda(d,P,\alpha,(0,h_0))\ge0$, then spreading happens for $(w,h)$.
 \end{proposition}

 To stress the dependence of the unique bounded positive solution $\theta(x)$ of \eqref{2.5} on parameter $\beta$, we rewrite $\theta$ as $\theta_{\beta}$.
Next we give two technical lemmas.

\begin{lemma}\label{l2.1}Let $c(d,P,\alpha,\beta,\mu)$ be given by \qq{x2.4}. Then $c(d,P,\alpha,\beta,\mu)$ is continuous and non-increasing in $\beta>0$.
\end{lemma}

\begin{proof}  Let $(w_{\beta},h_{\beta})$ be the unique solution of \eqref{2.4}. By Proposition \ref{p2.4}, we can let $h_0$ be large enough, depending only on the principal eigenvalue of \eqref{2.1}, such that spreading happens for $(w_{\beta},h_{\beta})$. Then in view of Proposition \ref{p2.4} again, we see $\lim_{t\to\yy}h_{\beta}(t)/t=c(d,P,\alpha,\beta,\mu)$. Owing to a comparison argument, we have $h_{\beta_1}(t)\ge h_{\beta_2}(t)$ for all $t\ge0$ and any $0<\beta_1<\beta_2$. Thus, $c(d,P,\alpha,\beta_1,\mu)\ge c(d,P,\alpha,\beta_2,\mu)$. The monotonicity follows.

Since the continuity can be derived by uniqueness and some compact methods, we omit the details. The proof is complete.
\end{proof}

\begin{lemma}\label{l2.2}Let $\theta_{\beta}$ be the unique bounded positive solution of \eqref{2.5}. Then in the space $L^{\yy}([0,\yy))$, $\theta_{\beta}$ is continuous,  and strictly decreasing with respect to $\beta>0$, i.e., $\lim_{\beta\to\beta_1}\|\theta_{\beta}-\theta_{\beta_1}\|_{L^\yy([0,\yy))}=0$ and  $\theta_{\beta_1}(x)>\theta_{\beta_2}(x)$ for $\beta_1<\beta_2$ and $x\ge0$.
\end{lemma}

\begin{proof} Choose $\beta_1<\beta_2$, and denote by $\theta_{\beta_1}$ and $\theta_{\beta_2}$ their corresponding bounded positive solutions, respectively. Then $\theta_{\beta_1}$ satisfies
\bess
d\int_{0}^{\yy}P(x-y)\theta_{\beta_1}(y)\dy-d\theta_{\beta_1}+\theta_{\beta_1}(\alpha-\beta_2\theta_{\beta_1})<0, ~ ~ x\in[0,\yy).
\eess
Then arguing as in the proof of \cite[Lemma 2.7]{LLW22}, we can deduce $\theta_{\beta_1}(x)\ge\theta_{\beta_2}(x)$ for $x\ge0$. Note that $\theta_{\beta_1}(\yy)=\alpha/\beta_1>\theta_{\beta_2}(\yy)=\alpha/\beta_2$. If there exists some $x_0\ge0$ such that $\theta_{\beta_1}(x_0)=\theta_{\beta_2}(x_0)$, then one can find the largest $x_1>0$ such that $\theta_{\beta_1}(x_1)=\theta_{\beta_2}(x_1)$ and $\theta_{\beta_1}(x)>\theta_{\beta_2}(x)$ for all $x\ge x_1$. Making use of the identities of $\theta_{\beta_i}$ with $i=1,2$, we have
\[d\int_{0}^{\yy}P(x_1-y)\theta_{\beta_1}(y)\dy-\beta_1\theta^2_{\beta_1}(x_1)=d\int_{0}^{\yy}P(x_1-y)\theta_{\beta_2}(y)\dy-\beta_2\theta^2_{\beta_2}(x_1),\]
which, combined with $\theta_{\beta_1}(x)>\theta_{\beta_2}(x)$ for all $x\ge x_1$, yields a contradiction. Thus $\theta_{\beta_1}(x)>\theta_{\beta_2}(x)$ for $x\ge0$.

Next we prove the continuity. Let $\beta_n\nearrow\beta_0$, and denote their corresponding bounded positive solutions by $\theta_{\beta_n}$ and $\theta_{\beta_0}$, respectively. By monotonicity, we see $\theta_{\beta_n}(x)>\theta_{\beta_0}(x)$ and $\theta_{\beta_n}(x)$ is strictly decreasing in $n\ge1$ for all $x\ge0$. Thus $\theta_{\yy}:=\lim_{n\to\yy}\theta_{\beta_n}(x)$ is well defined for $x\ge0$ and $\theta_{\yy}\ge\theta_{\beta_0}$. By the dominated convergence theorem and the uniqueness of the bounded positive solution (cf. \cite{LLW22}), we see $\theta_{\yy}=\theta_{\beta_0}$. In light of Dini's theorem, $\theta_{\beta_n}\to\theta_{\beta_0}$ locally uniformly in $[0,\yy)$ as $n\to\yy$. Obviously, for any small $\ep>0$, there exists a large $N_1$ such that $\alpha/\beta_n<\alpha/\beta_0+\ep$ for $n\ge N_1$. Together with the monotonicity of $\theta_{\beta}$ on $\beta$, one can find a large $X\gg1$ such that $\theta_{\beta_n}(x)\le\theta_{\beta_0}(x)+\ep$ for $x\ge X$ and $n\ge N_1$. As  $\theta_{\beta_n}\to\theta_{\beta_0}$ uniformly in $[0,X]$, we can find $N_2$ such that $\theta_{\beta_n}(x)\le\theta_{\beta_0}(x)+\ep$ for $x\in[0,X]$ and $n\ge N_2$. Setting $N=\max\{N_1,N_2\}$, one sees $\theta_{\beta_n}(x)\le\theta_{\beta_0}(x)+\ep$ for $x\ge0$ and $n\ge N$. Similarly, we can handle the case $\beta_n\searrow\beta_0$. Therefore, in the space $C([0,\yy))$, $\theta_{\beta}$ is continuous with respect to $\beta>0$.
\end{proof}

Let $\oo=(l_1,l_2)$, with $l_1<l_2$, be an interval and consider the following steady state problem
\bes\label{2.6}
\begin{cases}
\dd d_1\int_\oo J_1(x-y)u(y)\dy-d_1u+f_1(u,v)=0, ~ ~ x\in\boo,\\[3mm]
\dd d_2\int_\oo J_2(x-y)v(y)\dy-d_2v+f_2(u,v)=0, ~ ~ x\in\boo.
\end{cases}
\ees

\begin{lemma}\lbl{le2.7}{\rm(Strong maximum principle)} Let $p\in L^\yy(\oo)$ and $0\le w\in L^\yy(\oo)$ satisfy
 \bes
 \int_\oo J(x-y)w(y)\dy-dw(x)+p(x)w(x)\le 0,\;\;x\in\boo,
  \lbl{3.4}\ees
where $d>0$. Then either $w\equiv 0$, or $w>0$ in $\boo$ and $\inf_{\oo}w>0$.
\end{lemma}

\begin{proof} Clearly, the function $h(x)=\int_\oo J(x-y)w(y)\dy$ is nonnegative and continuous in $\boo$. If there is $x_0\in\boo$ such that $h(x_0)=0$, then
the set
 \[ {\cal O}=\{x\in\boo:\, h(x)>0\}\]
is an open subset of $\boo$ by the continuity of $h(x)$. We can also observe that $w(x)>0$ in ${\cal O}$.

Let $x_i\in{\cal O}$ and $x_i\to \bar x\in\boo$, then $w(x_i)>0$. If $h(\bar x)=0$, then $w=0$ in a neighborhood $B_\sigma(\bar x)\cap\boo$ for some $\sigma>0$ as $J(0)>0$ and $J(\bar x-y)$ is continuous in $y\in\boo$. There exist $\tau >0$ and $i$ large enough such that $B_{\tau}(x_i)\subset B_{\sigma}(\bar{x})$, that is, $h>0$ in $B_{\tau}(x_i)\cap\boo$, which implies that $w>0$ in $B_{\tau}(x_i)\cap\boo$.  This contradiction yields that $h(\bar x)>0$. Thus, ${\cal O}$ is a closed subset of $\boo$. So ${\cal O}=\emptyset$ or ${\cal O}=\boo$,  then either $w\equiv 0$ or $w>0$ in $\boo$.

Assume that $w>0$ in $\oo$ and $\inf_{\oo}w=0$. Then there exist $x_k\in\oo$ and $x_0\in\boo$ such that $x_k\to x_0$ and $w(x_k)\to 0$. Hence, $h(x_k)\to h(x_0)>0$ and $p(x_k)w(x_k)\to 0$. Taking $x=x_k$ in \qq{3.4} and letting $k\to\yy$ we can get a contradiction.
\end{proof}

\begin{lemma}\lbl{le2.8} Let $(u,v)\in [L^\yy(\oo)]^2$ be a nonnegative solution of \qq{2.6}. Then one of the followings holds:\vspace{-2mm}
\begin{enumerate}[$(1)$]
  \item $u\equiv 0$ and $v\equiv0$ in $\boo$;\vspace{-2.5mm}
 \item $u\equiv 0$ and $v>0$ in $\boo$, and $v\in C(\boo)$;\vspace{-2.5mm}
 \item $v\equiv 0$ and $u>0$ in $\boo$, and $u\in C(\boo)$;\vspace{-2.5mm}
 \item $u>0, v>0$ in $\boo$, and $u,v\in C(\boo)$.
 \end{enumerate}\vspace{-2mm}
Moreover, problem \qq{2.6} has at most one positive solution.
\end{lemma}

\begin{proof} If $u\equiv 0$, then $v$ satisfies
 \[d_2\int_\oo J_2(x-y)v(y)\dy-d_2v+r_2v(1-2v)=0, ~ ~ x\in\boo.\]
Firstly, in view of Lemma \ref{le2.7} it yields that either $v\equiv 0$ in $\boo$, or $v>0$ in $\boo$ and $\inf_\oo v>0$. In the latter case, it is easy to see that $v\in C(\boo)$.

If $u\not\equiv 0$, $v\not\equiv 0$, then $u, v>0$ in $\boo$ and $\inf_\oo u, \inf_\oo v>0$ by Lemma \ref{le2.7}.

We will prove $u,v\in C(\boo)$ by using the implicit function theorem.  Let
 \bess
&\dd Q_1(x)=d_1\int_\oo J_1(x-y)u(y)\dy, ~  ~ Q_2(x)=d_2\int_\oo J_2(x-y)v(y)\dy,\\[1mm]
&P_1(x,w,z)=Q_1(x)-d_1w+f_1(w,z), ~ ~ P_2(x,w,z)=Q_2(x)-d_2z+f_2(w,z).
 \eess
 Clearly, $P_i$ for $i=1,2$ are continuous with respect to $x\in\boo$ and $ w,z\ge0$, and $(u,v)$ satisfies
 \[P_1(x,u,v)=0,\;\;\;P_2=(x,u,v)=0,\;\;\;x\in\boo.\]
By the direct calculations, we have
\bess
{\rm det}\kk.\frac{\partial (P_1,P_2)}{\partial (u,v)}\rr|_{(u,v)}&=&\kk(ar_1-d_1-2r_1u-\frac{2r_1u}{1+bv}\rr)
\kk(r_2-d_2-2r_2v-\frac{2r_2v}{1+qu}\rr)-\frac{r_1r_2u^2v^2bq}{(1+bv)^2(1+qu)^2}\\[2mm]
&=&\kk(\frac{Q_1}{u}+r_1u+\frac{r_1u}{1+bv}\rr)
\kk(\frac{Q_2}{v}+r_2v+\frac{r_2v}{1+qu}\rr)-\frac{r_1r_2u^2v^2bq}{(1+bv)^2(1+qu)^2}\\[2mm]
&\ge&\kk(\frac{Q_1}{u}+r_1u+\frac{r_1u}{1+bv}\rr)
\kk(\frac{Q_2}{v}+r_2v+\frac{r_2v}{1+qu}\rr)-\frac{r_1ur_2v}{(1+bv)(1+qu)}>0.
\eess
It follows from the implicit function theorem that $(u,v)$ is continuous in $x\in\boo$.

At last we prove that \qq{2.6} has at most one positive solution. Let $(u_i,v_i)$, $i=1,2$, be two positive solutions. Then $u_i, v_i\in C(\boo)$ by the conclusion  (4). Thus, we can define
\[k^*=\inf\{k\ge1: k(u_1(x),v_1(x))\ge(u_2(x),v_2(x)) ~ ~ {\rm for ~ }x\in\boo\}.\]
Clearly, $k^*\ge1$ and $k^*(u_1(x),v_1(x))\ge(u_2(x),v_2(x))$ in $\boo$. We now show $k^*=1$. Otherwise, there exists some $x_0\in\boo$ such that $k^*u_1(x_0)=u_2(x_0)$ or $k^*v_1(x_0)=v_2(x_0)$. Without loss of generality, we assume that $k^*u_1(x_0)=u_2(x_0)$. In view of the identity of $u_i$ for $i=1,2$, we have
\bess
u_1(x_0)\left(1+\frac{1}{1+bv_1(x_0)}\right)\ge u_2(x_0)\left(1+\frac{1}{1+bv_2(x_0)}\right),
\eess
which, combined with $k^*u_1(x_0)=u_2(x_0)$, yields
\bess
\frac{1}{k^*}+\frac{1}{k^*+bk^*v_1(x_0)}\ge 1+\frac{1}{1+bv_2(x_0)}.
\eess
Owing to $k^*v_1(x_0)\ge v_2(x_0)$, we have
\bess
\frac{1}{k^*}+\frac{1}{k^*+bv_2(x_0)}\ge 1+\frac{1}{1+bv_2(x_0)}.
\eess
This is a contradiction. So $k^*=1$, i.e.,  $(u_1(x),v_1(x))\ge(u_2(x),v_2(x))$ in $\boo$. Exchanging the positions of $(u_1,v_1)$ and $(u_2,v_2)$, we can deduce $(u_2(x),v_2(x))\ge(u_1(x),v_1(x))$ in $\boo$. The uniqueness follows, and the proof is complete. \end{proof}

Let
 \bess
 \lm_1^p(\oo):=\lambda(d_1,J_1,r_1a,\oo),\;\;\;
 \lm_2^p(\oo):=\lambda(d_2,J_2,r_2,\oo)\eess
be the principal eigenvalues of \eqref{2.1} with $(d,P,\alpha)$ replaced by $(d_1,J_1,r_1a)$ and $(d_2,J_2,r_2)$, respectively. Let $(u^*,v^*)$ be the unique positive root of
  \bes\label{1.2}
   f_1(u,v)=f_2(u,v)=0.\ees

 \begin{lemma}\label{l2.3}
\begin{enumerate}[$(1)$]
  \item If $\lm_i^p(\oo)\le0$, $i=1,2$, then $(0,0)$ is the unique nonnegative solution of \eqref{2.6}.
\item If $\lm_1^p(\oo)>0$ and $\lm_2^p(\oo)\le0$, then \eqref{2.6} only has two nonnegative solutions, $(0,0)$ and $(\theta_1,0)$, where $\theta_1\in C(\boo)$ is the unique positive solution of
      \bes\label{x2.1}
      d_1\int_\oo J_1(x-y)\theta(y)\dy-d_1\theta+r_1\theta(a-2\theta)=0, ~ ~ x\in\boo.
      \ees
\item If  $\lm_1^p(\oo)\le0$ and $\lm_2^p(\oo)>0$, then \eqref{2.6} only has two nonnegative solutions, $(0,0)$ and $(0,\theta_2)$, where $\theta_2\in C(\boo)$ is the unique positive solution of
      \bes\label{a2.2}
      d_2\int_\oo J_2(x-y)\theta(y)\dy-d_2\theta+r_2\theta(1-2\theta)=0, ~ ~ x\in\boo.
      \ees
\item If $\lm_i^p(\oo)>0$, $i=1,2$, then \eqref{2.6} has a trivial solution $(0,0)$, two semi-trivial solutions $(\theta_1,0)$ and $(0,\theta_2)$, and a unique positive solution $(u_\oo,v_\oo)\in [C(\boo)]^2$, where $\theta_1$ and $\theta_2$ are given by $(2)$ and $(3)$, respectively.\vspace{-2mm}
\item Take $\oo=(-l,l)$ with $l>0$, and denote $(u_l,v_l)=(u_\oo,v_\oo)$. Assume $\lm_1^p(\oo)>0$ and $\lm_2^p(\oo)>0$. Then $(u_l,v_l)$ is nondecreasing in $l$, and $(u_l,v_l)\to(u^*,v^*)$ locally uniformly in $\mathbb{R}$ as $l\to\yy$, where $(u^*,v^*)$ is given by \eqref{1.2}.
\end{enumerate}
\end{lemma}

\begin{proof}(1) Clearly, $(0,0)$ is a solution of \eqref{2.6}. If $(u,v)$ is a nonnegative solution of \eqref{2.6} and $u\not\equiv 0$, then $u\in C(\boo)$ and $u>0$ in $\boo$ by Lemma \ref{le2.7}. Thus we have
 \bess
 0&=&\dd d_1\int_\oo J_1(x-y)u(y)\dy-d_1u+f_1(u,v)\nm\\[1mm]
 &<&\dd d_1\int_\oo J_1(x-y)u(y)\dy-d_1u+r_1au,\;\;x\in\boo.
 \eess
It follows that $\lm_1^p(\oo)>0$. This is a contradiction. Similarly, $v\equiv0$.

(2) If $(u,v)$ is a nonnegative continuous solution of \eqref{2.6}, then  $v\equiv 0$ by the result (1). As $\lm_1^p(\oo)>0$, problem \eqref{x2.1}
has a unique positive solution $\theta_1$ and $\theta_1\in C(\boo)$.

The proof of (3) is similar.

(4) Due to $\lm_1^p(\oo)>0$ and $\lm_2^p(\oo)>0$, problems \qq{x2.1} and \eqref{a2.2}
have unique positive solutions $\theta_1$ and $\theta_2$, respectively, and $\theta_i\in C(\boo)$. Moreover, $\theta_1<a/2$ and $\theta_2<1/2$ on $\boo$. Thus
 \bess
\begin{cases}
\dd d_1\int_\oo J_1(x-y)\theta_1(y)\dy-d_1\theta_1+f_1(\theta_1,\theta_2)>0, ~ ~ x\in\boo,\\[3mm]
\dd d_2\int_\oo J_2(x-y)\theta_2(y)\dy-d_2\theta_2+f_2(\theta_1,\theta_2)>0, ~ ~ x\in\boo.
\end{cases}
\eess
Let $(\ol u,\ol v)=(u^*,v^*)$ with $(u^*,v^*)$ uniquely given by \eqref{1.2}. Then $(\ol u,\ol v)>(\theta_1,\theta_2)$ and $(\ol u,\ol v)$ satisfies
 \bess
\begin{cases}
\dd d_1\int_\oo J_1(x-y)\ol u(y)\dy-d_1\ol u+f_1(\ol u,\ol v)\le0, ~ ~ x\in\boo,\\[3mm]
\dd d_2\int_\oo J_2(x-y)\ol v(y)\dy-d_2\ol v+f_2(\ol u,\ol v)\le0, ~ ~ x\in\boo.
\end{cases}
 \eess
By the upper and lower solutions method, \eqref{2.6} has at least one positive solution $(u_\oo,v_\oo)$ and $(\theta_1, \theta_2)\le(u_\oo,v_\oo)\le (\ol u,\ol v)$. By Lemma \ref{le2.8}, $(u_\oo,v_\oo)$ is the unique positive solution of \eqref{2.6} and $u_\oo, v_\oo\in C(\boo)$.

(5) We first show $(u_l,v_l)$ is nondecreasing for all large $l>0$. Let $l>k>0$. Then $(u_l,v_l)$ satisfies
\bess
\begin{cases}
\dd d_1\int_{-k}^{k}J_1(x-y)u_l(y)\dy-d_1u_l+f_1(u_l,v_l)\le0, ~ ~ x\in[-k,k],\\[3mm]
\dd d_2\int_{-k}^{k}J_2(x-y)v_l(y)\dy-d_2v_l+f_2(u_l,v_l)\le0, ~ ~ x\in[-k,k].
\end{cases}
\eess
Arguing as in the proof of Lemma \ref{le2.8}, we can show that $(u_l,v_l)\ge (u_k,v_k)$.

From the arguments in the proof of (4), we know the unique positive solution $(u_l,v_l)$ of \eqref{2.6} satisfies $(u_l,v_l)\le (u^*,v^*)$ in $\boo$. Thus $(\tilde u(x),\tilde v(x)):=\lim_{l\to\yy}(u_l(x),v_l(x))$ is well defined for all $x\in\mathbb{R}$, and $(\tilde u(x),\tilde v(x))$ is positive. Using the dominated convergence theorem, we see $(\tilde u,\tilde v)$ satisfies
\bess
\begin{cases}
\dd d_1\int_{-\yy}^{\yy}J_1(x-y)\tilde u(y)\dy-d_1\tilde u+f_1(\tilde u,\tilde v)=0, ~ ~ x\in\mathbb{R},\\[3mm]
\dd d_2\int_{-\yy}^{\yy}J_2(x-y)\tilde v(y)\dy-d_2\tilde v+f_2(\tilde u,\tilde v)=0, ~ ~ x\in\mathbb{R}.
\end{cases}
\eess
We now show $(\tilde u,\tilde v)\equiv(u^*,v^*)$. Obviously, it is sufficient to show that $\tilde u$ and $\tilde v$ are positive constants since system $f_1(u,v)=0$ and $f_2(u,v)=0$ only has a unique positive root $(u^*,v^*)$. For the given  $x_1\in\mathbb{R}\setminus\{0\}$. When $l>2|x_1|$, it is easy to verify that
\[[-l+2|x_1|,l-2|x_1|]\subset[-l+|x_1|-x_1,l-|x_1|-x_1]\subset[-l,l].\]
As in the proof of monotonicity, we can see that for $x\in[-l+2|x_1|,l-2|x_1|]$,
\[(u_{l-2|x_1|}(x),v_{l-2|x_1|}(x))\le(u_{l-|x_1|}(x+x_1),v_{l-|x_1|}(x+x_1))\le(u_{l}(x),v_{l}(x)).\]
Letting $l\to\yy$, we have $(\tilde u(x+x_1),\tilde v(x+x_1))=(\tilde u(x),\tilde v(x))$ for all $x\in\mathbb{R}$. Setting $x=0$, we get $(\tilde u(x_1),\tilde v(x_1))=(\tilde u(0),\tilde v(0))$. Thus the assertion (5) is obtained. The proof is finished.
\end{proof}

Next we consider the longtime behaviors of the following evolving problem on fixed domain
\bes\label{2.10}
\begin{cases}
\dd u_t=d_1\int_\oo J_1(x-y)u(t,y)\dy-d_1u+f_1(u,v), ~ ~ t>0, ~ x\in\boo,\\[3mm]
\dd v_t=d_2\int_\oo J_2(x-y)v(t,y)\dy-d_2v+f_2(u,v), ~ ~ t>0, ~ x\in\boo,\\[2mm]
(u(0,x),v(0,x))\in [C(\boo)]^2, ~ u(0,x)\ge,\not\equiv0, ~ v(0,x)\ge,\not\equiv0, ~ x\in\boo.
\end{cases}
\ees

\begin{proposition}\label{p2.5}Let $(u,v)$ be the unique solution of \eqref{2.10}, and $\theta_i(x)$ be given in Lemma \ref{l2.3}, $i=1,2$. Then the following statements hold.
\begin{enumerate}[$(1)$]
  \item If $\lm_1^p(\oo)\le0$ and $\lm_2^p(\oo)\le0$, then $\lim_{t\to\yy}(u(t,x),v(t,x))=(0,0)$ uniformly in $\boo$.
  \item If $\lm_1^p(\oo)>0$ and $\lm_2^p(\oo)\le0$, then $\lim_{t\to\yy}(u(t,x),v(t,x))=(\theta_1(x),0)$ uniformly in $\boo$.
  \item If $\lm_1^p(\oo)\le0$ and $\lm_2^p(\oo)>0$, then $\lim_{t\to\yy}(u(t,x),v(t,x))=(0,\theta_2(x))$ uniformly in $\boo$.
\item If $\lm_1^p(\oo)>0$ and $\lm_2^p(\oo)>0$, then $\lim_{t\to\yy}(u(t,x),v(t,x))=(u_\oo(x),v_\oo(x))$ uniformly in $\boo$, where $(u_\oo(x),v_\oo(x))$ is the unique positive solution of \eqref{2.6}.
\end{enumerate}
\end{proposition}

\begin{proof} Here we only prove the conclusion (4) as (1)-(3) are obviously.

Firstly, by the maximum principle, $(u(t,x),v(t,x))>(0,0)$ for $t>0$ and $x\in\boo$. We may assume $(u(0,x),v(0,x))>(0,0)$ for $x\in\boo$.

Let $\phi_1$ and $\phi_2$ be the corresponding positive eigenfunctions of $\lm_1^p(\oo)$ and $\lm_2^p(\oo)$, respectively. It is easy to verify that for small $\ep>0$, there holds
 \bess\begin{cases}
\dd d_1\int_\oo J_1(x-y)\ep\phi_1(y)\dy-d_1\ep\phi_1+f_1(\ep\phi_1,\ep\phi_2)\ge0, ~ ~ x\in\boo,\\[3mm]
\dd d_2\int_\oo J_2(x-y)\ep\phi_2(y)\dy-d_2\ep\phi_2+f_2(\ep\phi_1,\ep\phi_2)\ge0, ~ ~ x\in\boo.
 \end{cases}\eess
Let $0<\ep\ll 1$ and $K_1,K_2\gg 1$ such that $\ep(\phi_1(x),\phi_2(x))\le (u(0,x),v(0,x))\le (K_1,K_2)$, and $f(K_1,K_2)\le0$ and $f_2(K_1,K_2)\le0$. Let $(\bar{u},\bar{v})$ and $(\ud u,\ud v)$ be the solutions of \eqref{2.10} with  initial functions $(K_1,K_2)$ and $\ep(\phi_1,\phi_2)$, respectively. In light of some comparison arguments, $(\ud u,\ud v)\le (u,v)\le(\bar{u},\bar{v})$ in $\boo\times[0,\yy)$, and $(\ud u,\ud v)$ and $(\bar{u},\bar{v})$ are nondecreasing and nonincreasing in $t$, respectively. It is easy to see that $(\ud u(t,x),\ud v(t,x))$ and $(\bar{u}(t,x),\bar{v}(t,x))$ converge to $(u_\oo(x),v_\oo(x))$ uniformly in $\boo$ as $t\to\yy$. Hence $(u(t,x),v(t,x))\to(u_\oo(x),v_\oo(x))$ uniformly in $\boo$ as $t\to\yy$. The proof is ended.
\end{proof}

Next we consider the following steady state problem on half space $[0,\yy)$
 \bes\label{2.11}
\begin{cases}
\dd d_1\int_{0}^{\yy}J_1(x-y)u(y)\dy-d_1u+f_1(u,v)=0, ~ ~ x\in[0,\yy),\\[3mm]
\dd d_2\int_{0}^{\yy}J_2(x-y)v(y)\dy-d_2v+f_2(u,v)=0, ~ ~ x\in[0,\yy).
 \end{cases}\ees

Let $\theta_{1\yy}$ and $\theta_{2\yy}$ be the unique bounded positive solutions of \eqref{2.5} with $(d,P,\alpha,\beta)$ replaced by $(d_1,J_1,r_1a, 2r_1)$ and $(d,P,\alpha,\beta)=(d_2,J_2,r_2, 2r_2)$, respectively.

\begin{theorem}\label{p2.6}Problem \eqref{2.11} has a trivial solution $(0,0)$, two semi-trivial continuous nonnegative solutions $(\theta_{1\yy},0)$ and $(0,\theta_{2\yy})$, and a unique positive solution $(\tilde u,\tilde v)$, where $\theta_{i\yy}$ for $i=1,2$ are uniquely given as above. Moreover, this positive solution $(\tilde u,\tilde v)$ is continuous and strictly increasing in $[0,\yy)$, and $\lim_{x\to\yy}(\tilde u(x),\tilde v(x))=(u^*,v^*)$.
\end{theorem}
\begin{proof}Certainly, we only need to show the existence and uniqueness of the positive solution and its properties.

{\it Step 1: The existence}. By \qq{x2.2}, there exists $l_0\gg 1$ such that $\lambda(d_1,J_1,r_1a,(0,l))>0$ and $\lambda(d_2,J_2,r_2,(0,l))>0$ for $l\ge l_0$. By  virtue of Lemma \ref{l2.3}, the problem
\bess
\begin{cases}
\dd d_1\int_{0}^{l}J_1(x-y)u(y)\dy-d_1u+f_1(u,v)=0, ~ ~ x\in[0,l],\\[2mm]
\dd d_2\int_{0}^{l}J_2(x-y)v(y)\dy-d_2v+f_2(u,v)=0, ~ ~ x\in[0,l]
\end{cases}
\eess
has a unique positive solution $(u_l,v_l)$, which is continuous for $x\in[0,l]$ and $(u_l,v_l)\le (u^*,v^*)$. Moreover, as in the proof of (5) of Lemma \ref{l2.3}, we can show that $(u_l,v_l)$ is nondecreasing for all large $l>0$. So we can define $(\tilde u(x),\tilde v(x))=\lim_{l\to\yy}(u_l(x),v_l(x))$ for $x\ge0$, and clearly $(0,0)<(\tilde u(x),\tilde v(x))\le(u^*,v^*)$ for all $x\ge0$. Due to the dominated convergence theorem, $(\tilde u,\tilde v)$ satisfies \eqref{2.11}.

Similar to the arguments in the proof of Lemma \ref{le2.8}, we can show $(\tilde u(x),\tilde v(x))$ is continuous in $x\ge0$. Now we show that $(\tilde u(x),\tilde v(x))\to(u^*,v^*)$ as $x\to\yy$. Otherwise, there exist $\ep_0>0$ and a sequence $\{x_n\}$ satisfying $x_n\nearrow\yy$ as $n\to\yy$ such that $\tilde u(x_n)\le u^*-\ep_0$ or $\tilde v(x_n)\le v^*-\ep_0$. Without loss of generality, we suppose $\tilde u(x_n)\le u^*-\ep_0$ for $n\ge1$.

Set $w_n(x)=u_{2x_n}(x+x_n)$ and $z_n(x)=v_{2x_n}(x+x_n)$. It is easy to see that $(w_n,z_n)$ satisfies
  \bess\begin{cases}
 \dd d_1\int_{-x_n}^{x_n}J_1(x-y)w_n(y)\dy-d_1w_n+f_1(w_n,z_n)=0, ~ ~ x\in[-x_n,x_n],\\[3mm]
\dd d_2\int_{-x_n}^{x_n}J_2(x-y)z_n(y)\dy-d_2z_n+f_2(w_n,z_n)=0, ~ ~ x\in[-x_n,x_n].
 \end{cases}\eess
Thus $(w_n,z_n)$ is the unique positive solution of \eqref{2.6} on $[-x_n,x_n]$. It then follows from (5) of Lemma \ref{l2.3} that $(w_n(x),z_n(x))\to(u^*,v^*)$ locally uniformly in $\mathbb{R}$ as $n\to\yy$. Hence for large $n$,
\[(u_{2x_n}(x_n),v_{2x_n}(x_n))=(w_n(0),z_n(0))\ge(u^*-{\ep_0}/{2},v^*-{\ep_0}/{2}).\]
Due to the definition of $(\tilde u,\tilde v)$, we have
\[(\tilde u(x_n),\tilde v(x_n))\ge(u_{2x_n}(x_n),v_{2x_n}(x_n))\ge(u^*-{\ep_0}/{2},v^*-{\ep_0}/{2}).\]
 However, since $\tilde u(x_n)\le u^*-\ep_0$ for $n\ge1$, we derive a contradiction. So \bes
 (\tilde u(x),\tilde v(x))\to(u^*,v^*)\;\;\;\mbox{ as}\;\;x\to\yy.
 \lbl{x2.13}\ees

{\it Step 2; The uniqueness}. Let $(u,v)$ be another bounded positive solution of \eqref{2.11}. We first show that
 \bes
 (u(x),v(x))\to(u^*,v^*)
 \;\;\;\mbox{ as}\;\;x\to\yy.
 \lbl{x2.14}\ees

Analogously, we can show that $(u,v)$ is continuous in $x\ge0$. Moreover, by arguing as in the proof of (5) of Lemma \ref{l2.3}, we can conclude that $(u,v)\ge(u_l,v_l)$ for $x\in[0,l]$ and all large $l$, which, combined with the definition of $(\tilde u,\tilde v)$, yields that $(u,v)\ge(\tilde u,\tilde v)$.

We now show that $(u,v)\le(u^*,v^*)$. By way of contradiction, without loss of generality, we suppose that $u_{\rm sup}:=\sup_{[0,\yy)}u>u^*$. Then there are two cases that need to be considered:\vspace{-1mm}
\begin{enumerate}[$(1)$]
  \item {\it Case 1:} There exists some $x_1\in[0,\yy)$ such that $u(x_1)=u_{\rm sup}>u^*$.\vspace{-1mm}
  \item {\it Case 2:} $u(x)<u_{\rm sup}$ for all $x\ge0$, and there exists a sequence converging to $\yy$ such that $u$ converges to $u_{\rm sup}>u^*$ along this sequence.
\end{enumerate}\vspace{-1mm}

We next show both these two cases are impossible, which implies $u_{\rm sup}\le u^*$.
For Case 1, substituting $x_1$ into the identity of $u$ yields
\bes\label{2.12}
a-u(x_1)-\frac{u(x_1)}{1+bv(x_1)}\ge0.\ees
This indicates that $v(x_1)>v^*$, i.e., $v_{\rm sup}:=\sup_{[0,\yy)}v>v^*$. If there exists some $x_2\ge0$ such that $v(x_2)=\sup_{[0,\yy)}v>v^*$, then putting $x_2$ into the identity of $v$ leads to $1-v(x_2)-\frac{v(x_2)}{1+qu(x_2)}\ge0$. Notice that $u(x_1)\ge u(x_2)$ and $v(x_1)\le v(x_2)$. We can obtain
  \bess
  a-u(x_1)-\frac{u(x_1)}{1+bv(x_2)}\ge0 ~ ~ {\rm and ~ ~ }1-v(x_2)-\frac{v(x_2)}{1+qu(x_1)}\ge0.\eess
It is not hard to show that there exists another positive root of $f_1(u,v)=0$ and $f_2(u,v)$, which is different from $(u^*,v^*)$. This is a contradiction. So in Case 1, $v$ can not achieve its supremum that is large than $v^*$. If $v(x)<v_{\rm sup}$ for all $x\ge0$, then one can find a sequence $\{x_n\}\nearrow\yy$ such that $v(x_n)\to v_{\rm sup}$ as $n\to\yy$. Since $u(x_n)$ is bounded, there exists a subsequence, still denoted by itself, such that $u(x_n)$ converges to some positive constant $u_{\yy}$. Moreover, direct calculations show
\bess
\limsup_{n\to\yy}\kk(\int_{0}^{\yy}J_2(x_n-y)v(y)\dy
-v(x_n)\rr)\le\limsup_{n\to\yy}(v_{\rm sup}-v(x_n))=0,
\eess
which, combined with the identity of $v$, yields that
\[1-v_{\rm sup}-\frac{v_{\rm sup}}{1+qu_{\yy}}\ge0.\]
Together with \eqref{2.12}, we also can derive a contradiction. Thus Case 1 is impossible. Similarly, we can show Case 2 is also impossible, and the details are omitted here. Hence $(u,v)\le(u^*,v^*)$, which combined with $(u,v)\ge(\tilde u,\tilde v)$ gives \qq{x2.14}.

With the aid of \qq{x2.13} and \qq{x2.14}, we can argue as in the proof of Lemma \ref{le2.8} to show the uniqueness. Since the modifications are obvious, we omit the details.

{\it Step 3: The monotonicity of $(\tilde u,\tilde v)$}. For any $x\ge0$ and $\delta>0$, denote $(\tilde u(x+\delta),\tilde v(x+\delta))$ by $(u^{\delta}_{\yy}(x),v^{\delta}_{\yy}(x))$. Then $(u^{\delta}_{\yy}(x),v^{\delta}_{\yy}(x))$ satisfies
 \bess\begin{cases}
 \dd d_1\int_{-\delta}^{\yy}J_1(x-y)u^{\delta}_{\yy}(y)\dy-d_1u^{\delta}_{\yy}
 +f_1(u^{\delta}_{\yy},v^{\delta}_{\yy})=0, ~ ~ x\in[0,\yy),\\[3mm]
\dd d_2\int_{-\delta}^{\yy}J_2(x-y)v^{\delta}_{\yy}(y)\dy-d_2v^{\delta}_{\yy}
+f_2(u^{\delta}_{\yy},v^{\delta}_{\yy})=0, ~ ~  x\in[0,\yy),
 \end{cases}\eess
which implies
 \bess\begin{cases}
 \dd d_1\int_{0}^{\yy}J_1(x-y)u^{\delta}_{\yy}(y)\dy-d_1u^{\delta}_{\yy}
 +f_1(u^{\delta}_{\yy},v^{\delta}_{\yy})\le0, ~ ~ x\in[0,\yy),\\[3mm]
\dd d_2\int_{0}^{\yy}J_2(x-y)v^{\delta}_{\yy}(y)\dy-d_2v^{\delta}_{\yy}
+f_2(u^{\delta}_{\yy},v^{\delta}_{\yy})\le0, ~ ~  x\in[0,\yy).
\end{cases}
\eess
Then we can argue as in the proof of Lemma \ref{le2.8} to show that $(u^{\delta}_{\yy}(x),v^{\delta}_{\yy}(x))\ge(\hat{u}(x),\hat{v}(x))$ for all $x\ge0$. Thus $(\tilde u(x),\tilde v(x))$ is nondecreasing in $x\ge0$.

Now we show that $(\tilde u,\tilde v)$ is strictly increasing in $x\ge0$. We only prove the monotonicity of $\tilde u$ since one can similarly prove the case for $\tilde v$. Thanks to {\bf (J)}, there exists a $\delta>0$ such that $J_1(x)>0$ for $x\in[-\delta,\delta]$. Hence it is sufficient to show that $\tilde u$ is strictly increasing in $[k\delta,(k+1)\delta]$ for all nonnegative integer $k$. If there exist $x_1$ and $x_2$ with $0\le x_1<x_2\le \delta$ such that $\tilde u(x_1)=\tilde u(x_2)$, then by the equation of $\tilde u$, we have
\bess
d_1\int_{0}^{\yy}J_1(x_1-y)\tilde u(y)\dy-\frac{r_1\tilde u^2(x_1)}{1+b\tilde v(x_1)}=d_1\int_{0}^{\yy}J_1(x_2-y)\tilde u(y)\dy-\frac{r_1\tilde u^2(x_1)}{1+b\tilde v(x_2)},
\eess
which, combined with $\tilde v(x_1)\le \tilde v(x_2)$, yields
\bes\label{2.14}
\int_{0}^{\yy}J_1(x_1-y)\tilde u(y)\dy\ge\int_{0}^{\yy}J_1(x_2-y)\tilde u(y)\dy.\ees
However, a straightforward computation gives
\bess
&&\int_{0}^{\yy}\!J_1(x_2-y)\tilde u(y)\dy-\int_{0}^{\yy}\!J_1(x_1-y)\tilde u(y)\dy\\[1mm]
&=&\int_{-x_2}^{\yy}\!J_1(y)\tilde u(y+x_2)\dy-\int_{-x_1}^{\yy}\!J_1(y)\tilde u(y+x_1)\dy\\[1mm]
&=&\int_{-x_2}^{-x_1}\!J_1(y)\tilde u(y+x_2)\dy+\int_{-x_1}^{\yy}\!J_1(y)\tilde u(y+x_2)\dy
-\int_{-x_1}^{\yy}\!J_1(y)\tilde u(y+x_1)\dy\\[1mm]
&\ge&\int_{-x_2}^{-x_1}\!J_1(y)\tilde u(y+x_2)\dy>0.
\eess
So this contradiction shows that $\tilde u$ is strictly increasing in $[0,\delta]$.

Arguing inductively, we assume that $\tilde u$ is strictly increasing in $[k\delta,(k+1)\delta]$. If there exist $x_1$ and $x_2$ with $(k+1)\delta\le x_1<x_2\le(k+2)\delta$ such that $\tilde u(x_1)=\tilde u(x_2)$, then by the equation of $\tilde u$, we see \eqref{2.14} holds. Moreover, since $J_1(x)>0$ for $x\in[-\delta,\delta]$ and $\tilde u(x_2-\delta)-\tilde u(x_1-\delta)>0$, we have
\bess
&&\int_{0}^{\yy}\!J_1(x_2-y)\tilde u(y)\dy-\int_{0}^{\yy}\!J_1(x_1-y)\tilde u(y)\dy\\[1mm]
&=&\int_{-x_2}^{\yy}\!J_1(y)\tilde u(y+x_2)\dy-\int_{-x_1}^{\yy}\!J_1(y)\tilde u(y+x_1)\dy\\[1mm]
&=&\int_{-x_2}^{-x_1}\!J_1(y)\tilde u(y+x_2)\dy+\int_{-x_1}^{\yy}\!J_1(y)\tilde u(y+x_2)\dy
-\int_{-x_1}^{\yy}\!J_1(y)\tilde u(y+x_1)\dy\\[1mm]
&\ge&\int_{-\delta}^{0}\!J_1(y)(\tilde u(y+x_2)-\tilde u(y+x_1))\dy>0.
 \eess
As above, we can get a contradiction, and thus $\tilde u$ is strictly increasing in $[0,\yy)$. The proof is complete.
\end{proof}

\section{Longtime behaviors of \eqref{1.1}}
In this section, we investigate the longtime behaviors of \eqref{1.1}. Firstly, taking advantage of the similar arguments as in the proofs of \cite{DN20,NV,LWWcpaa}, we can prove \eqref{1.1} admits a unique global solution $(u,v,s_1,s_2)$. Moreover,  for any $T>0$, $(u,v,s_1,s_2)\in C(D^T_1)\times C(D^T_2)\times [C^1([0,T])]^2$ where $D^T_i:=\{(t,x):0\le t\le T, ~ 0\le x\le s_i(t)\}$ for $i=1,2$, $(0,0)\le(u(t,x),v(t,x))\le (\max\{\|u_0\|_{\yy},u^*\},\max\{\|v_0\|_{\yy},v^*\})$, and $s'_i(t)>0$ for $t>0$ and $x\in[0,s_i(t)]$ with $i=1,2$. Hence $s_i(\yy):=\lim_{t\to\yy}s_i(t)$ for $i=1,2$ are well defined, and $s_i(\yy)\in(s_{i0},\yy]$. If $s_1(\yy)=\yy$ ($s_2(\yy)=\yy$), we call the spreading of $u$ ($v$); otherwise, we call vanishing.

In the following, we study the longtime behaviors of \eqref{1.1} under spreading and vanishing cases. If $J_i$ satisfies {\bf (J1)} for $i=1,2$, we denote by
 \bes\begin{cases}
 \ud c_1:=c(d_1,J_1,r_1a,2r_1,\mu_1),\;\;\bar c_1:=c(d_1,J_1,r_1a,r_1,\mu_1),\\
 \ud c_2:=c(d_2,J_2,r_2,2r_2,\mu_2),\;\;\bar c_2:=c(d_2,J_2,r_2,r_2,\mu_2)
 \end{cases}\lbl{x3.1}\ees
the speeds of semi-waves of \eqref{2.2} with $(d,P,\alpha,\beta,\mu)$ replaced by  $(d_1,J_1,r_1a,2r_1,\mu_1)$, $(d_1,J_1,r_1a,r_1,\mu_1)$, $(d_2,J_2,r_2,2r_2,\mu_2)$ and $(d_2,J_2,r_2,r_2,\mu_2)$, respectively. For clarity, we denote
 \bess
 \lm_1^p(l):=\lambda(d_1,J_1,r_1a,\oo),\;\;\;
 \lm_2^p(l):=\lambda(d_2,J_2,r_2,\oo)\;\;\;\mbox{when}\;\;\oo=(0,l).\eess

\begin{theorem}\label{t3.1} Let $(u,v,s_1,s_2)$ be the solution of \eqref{1.1}, and $\theta_{i\yy}$ be given in Theorem \ref{p2.6}, $i=1,2$.
\begin{enumerate}[$(1)$]
  \item If $s_1(\yy)<\yy$ $(s_2(\yy)<\yy)$, then $\lm_1^p(s_1(\yy))\le0$ $(\lm_2^p(s_2(\yy))\le0)$. Moreover, $\lim_{t\to\yy}u(t,x)=0$ uniformly in $[0,\yy)$ if $s_1(\yy)<\yy$, and  $\lim_{t\to\yy}v(t,x)=0$ uniformly in $[0,\yy)$ if $s_2(\yy)<\yy$.
\item Suppose $s_1(\yy)=\yy$ and $s_2(\yy)<\yy$.
  \begin{enumerate}[{\rm(i)}]
    \item If $J_1$ satisfies {\bf (J1)}, then
    \bess
    \dd\lim_{t\to\yy}\frac{s_1(t)}{t}=\ud c_1,\; {\rm ~ and ~ }\,\lim_{t\to\yy}\max_{x\in[0,\,ct]}|u(t,x)-\theta_{1\yy}(x)|=0,
    \;\;\forall\;c\in[0,\ud c_1).
    \eess
\item If $J_1$ violates {\bf (J1)}, then
    \bess
     \dd \lim_{t\to\yy}\frac{s_1(t)}{t}=\yy,\; {\rm ~ and ~ }\lim_{t\to\yy}\max_{x\in[0,\,ct]}|u(t,x)-\theta_{1\yy}(x)|=0, \;\; \forall\; c\ge0.
    \eess
  \end{enumerate}
  \item Suppose $s_1(\yy)<\yy$ and $s_2(\yy)=\yy$. Then we have
  \begin{enumerate}[{\rm(i)}]
\item if $J_2$ satisfies {\bf (J1)}, then
  \bess
  \lim_{t\to\yy}\frac{s_2(t)}{t}=\ud c_2,\; {\rm ~ and ~ }\lim_{t\to\yy}\max_{x\in[0,ct]}|v(t,x)-\theta_{2\yy}(x)|=0,
  \;\;\forall\;c\in[0,\ud c_2).
  \eess
 \item if $J_2$ violates {\bf (J1)}, then
   \bess
 \lim_{t\to\yy}\frac{s_2(t)}{t}=\yy,\; {\rm ~ and ~ }\lim_{t\to\yy}\max_{x\in[0,ct]}|v(t,x)-\theta_{2\yy}(x)|=0, \;\;\forall\;c\ge0.
    \eess
  \end{enumerate}
 \item Suppose $s_1(\yy)=\yy$ and $s_2(\yy)=\yy$. Then $\lim_{t\to\yy}(u(t,x),v(t,x))=(\tilde u(x),\tilde v(x))$ locally uniformly in $[0,\yy)$, where $(\tilde u,\tilde v)$ is the unique bounded positive solution of \eqref{2.11}. Moreover, we have the following estimates for the spreading speeds of free boundaries $s_i(t)$ with $i=1,2$.
      \begin{enumerate}[{\rm(i)}]
 \item If $J_1$ satisfies {\bf (J1)}, then
      \bess
      \ud c_1\le \liminf_{t\to\yy}\frac{s_1(t)}{t}\le \limsup_{t\to\yy}\frac{s_1(t)}{t}\le \bar c_1.
      \eess
 \item If $J_1$ violates {\bf (J1)}, then $\lim_{t\to\yy}\frac{s_1(t)}{t}=\yy$.
   \item If $J_2$ satisfies {\bf (J1)}, then
      \bess
 \ud c_2\le \liminf_{t\to\yy}\frac{s_2(t)}{t}\le \limsup_{t\to\yy}\frac{s_2(t)}{t}\le\bar c_2.
      \eess
\item If $J_2$ violates {\bf (J1)}, then $\lim_{t\to\yy}\frac{s_2(t)}{t}=\yy$.
      \end{enumerate}
\end{enumerate}
\end{theorem}

\begin{proof}

(1) We only prove the assertion for $u$ and $s_1$ since the other case can be dealt with similarly. Arguing on the contrary that $\lm_1^p(s_1(\yy))>0$, by continuity, there is $T>0$ such that $\lm_1^p(s_1(T))>0$. Clearly, $(u(t+T),s_1(t+T))$ is an upper solution of \eqref{2.4} with $(d,P,\alpha,\beta,\mu, h_0,w_0)$ replaced by $(d_1,J_1,r_1a,2r_1,\mu_1,s_1(T),u(T,x))$. Since $\lm_1^p(s_1(T))>0$, by Proposition \ref{p2.4}, spreading happens for \eqref{2.4}, i.e., $\lim_{t\to\yy}h(t)=\yy$. However, as $s_1(t+T)\ge h(t)$, we have $s_1(\yy)=\yy$. This is a contradiction. So $\lm_1^p(s_1(\yy))\le0$.

Consider problem
\bess
\begin{cases}
\dd\bar{u}_t=d_1\int_{0}^{s_1(\yy)}J_1(x-y)\bar{u}(t,y)\dy-d_1\bar{u}+r_1\bar{u}(a-\bar{u}), ~ ~ t>0, ~ x\in[0,s_1(\yy)],\\
\bar{u}(0,x)=\|u_0\|_{\yy}.
\end{cases}
\eess
Thanks to $\lm_1^p(s_1(\yy))\le0$, we have $\lim_{t\to\yy}\bar{u}(t,x)=0$ uniformly in $[0,s_1(\yy)]$. A comparison consideration yields $u(t,x)\le \bar{u}(t,x)$ for $t\ge0$ and $x\in[0,s_1(\yy)]$. Thus, $\lim_{t\to\yy}u(t,x)=0$ uniformly in $[0,\yy)$ as $u(t,x)=0$ for $x\ge s_1(t)$.

(2)  By virtue of $s_1(\yy)=\yy$ and \qq{x2.2}, we can find a large $T>0$ such that $\lm_1^p(s_1(T))>0$. Let $(\ud u,\ud s_1)$ be the unique solution of \eqref{2.4} with $(d,P,\alpha,\beta,\mu, h_0,w_0)=(d_1,J_1,r_1a,2r_1,\mu_1,s_1(T),u(T,x))$. In light of a comparison method, we see $(u(t+T,x),s_1(t+T))\ge(\ud u(t,x),\ud s_1(t))$ for $t\ge0$ and $x\in[0,\ud s_1(t)]$. By Proposition \ref{p2.4}, we have
  \bess
\lim_{t\to\yy}\frac{\ud s_1(t)}{t}=\ud c_1,\; {\rm ~ and ~ }\lim_{t\to\yy}\max_{x\in[0,ct]}|\ud u(t,x)-\theta_{1\yy}(x)|=0,\;\;\forall\;
c\in[0,\ud c_1).
 \eess
 Hence
\bes\label{3.1}
\liminf_{t\to\yy}\frac{ s_1(t)}{t}\ge\ud c_1,\; {\rm ~ and ~ }\liminf_{t\to\yy}u(t,x)\ge\theta_{1\yy}(x)\;\;\mbox{
uniformly in}\;\; x\in[0,ct],\,\,\forall\;c\in[0,\ud c_1).\ees

On the other hand, owing to $s_2(\yy)<\yy$ and the statement (1), $\lim_{t\to\yy}v(t,x)=0$ uniformly in $[0,\yy)$. For the small $\ep>0$, there exists a large $T>0$ such that $v(t,x)\le \ep$ for $t\ge T$ and $x\ge0$. Hence, $(u,s_1)$ satisfies
 \bess
\left\{\begin{aligned}
 & u_t\le d_1\dd\int_{0}^{s_1(t)}J_1(x-y)u(t,y)\dy-d_1 u+r_1u\kk(a-u-\frac{u}{1+b\ep}\rr), && t>T,~x\in[0,s_1(t)),\\
&u(t,x)=0, && t>T,~ x\ge s_1(t),\\
&s'_1(t)=\mu_1\dd\int_{0}^{s_1(t)}\!\!\int_{s_1(t)}^{\infty}
J_1(x-y)u(t,x)\dy\dx, && t>T,\\
&s_1(T)>0,\;\;  u(T,x)>0, &&x\in[0,s_1(T)).
\end{aligned}\right.
 \eess
 Let $(\bar{u},\bar{s}_1)$ be the unique solution of \eqref{2.4} with $(d_1,J_1,r_1a,\frac{2+b\ep}{1+b\ep}r_1,\mu_1,s_1(T),u(T,x))$ in place of $(d,P,\alpha,\beta,\mu,h_0,w_0)$.
 Clearly, $(u(t+T,x),s_1(t+T))\le(\bar{u}(t,x),\bar{s}_1(t))$ for $t\ge0$ and $x\in[0,s_1(t+T)]$. Together with $s_1(\yy)=\yy$, we have $\bar{s}_1(t)\to\yy$ as $t\to\yy$, which implies that spreading happens for $(\bar{u},\bar{s}_1)$. It then follows from Proposition \ref{p2.4} that
 \bes\label{3.2}
 \lim_{t\to\yy}\frac{\bar{s}_1(t)}{t}=c\kk(d_1,\,J_1,\,r_1a,\,r_1\frac{2+b\ep}{1+b\ep},\,\mu_1\rr) \; {\rm ~ and ~ }\;\lim_{t\to\yy}\max_{x\in[0,ct]}|\bar{u}(t,x)-\theta_{1\ep}(x)|=0,
 \ees
where $0\le c<c(d_1,J_1,r_1a,r_1\frac{2+b\ep}{1+b\ep},\mu_1)$ and $\theta_{1\ep}$ is the unique bounded positive solution of \eqref{2.5} with $(d,P,\alpha,\beta)$ replaced by $(d_1,J_1,r_1a, r_1\frac{2+b\ep}{1+b\ep})$. By Lemma \ref{l2.1} and the arbitrariness of $\ep$, we have $\limsup_{t\to\yy}s_1(t)/t\le\ud c_1$. Together with \eqref{3.1}, we obtain $\lim_{t\to\yy}s_1(t)/t=\ud c_1$.

Using Lemma \ref{l2.1}, \eqref{3.2} and $u(t+T,x)\le \bar{u}(t,x)$ for $t\ge0$ and $x\in[0,s_1(t)]$, we have
  \bes\label{3.3}\limsup_{t\to\yy}u(t,x)\le \theta_{1\ep}(x) ~ ~ {\rm uniformly ~ in ~ }[0,ct],\;\;\forall\;c\in[0,\ud c_1).\ees
By Lemma \ref{l2.2}, for the given $0<\delta\ll 1$, we can take $0<\ep\ll 1$ such that $\theta_{1\ep}\le \theta_{1\yy}+\delta$ for $x\ge0$, which, combined with \eqref{3.3}, yields that $\limsup_{t\to\yy}u(t,x)\le \theta_{1\yy}(x)+\delta$ uniformly in $[0,ct]$ for any given $c\in[0,\ud c_1)$. The arbitrariness of $\delta$ implies that $\limsup_{t\to\yy}u(t,x)\le \theta_{1\yy}(x)$ uniformly in $[0,ct]$ for any given $c\in[0,\ud c_1)$. Together with \eqref{3.1}, we complete the proof of the assertion (i) in (2).

If $J_1$ does not satisfies {\bf (J1)}, then it follows from Proposition \ref{p2.4} that $\lim_{t\to\yy}{\ud s_1(t)}/{t}=\yy$ and $\lim_{t\to\yy}\max_{x\in[0,\,ct]}|\ud u(t,x)-\theta_{1\yy}(x)|=0$ for any $c\ge0$. Hence
  \bess
\lim_{t\to\yy}\frac{ s_1(t)}{t}=\yy,\;  {\rm ~ and ~ }\liminf_{t\to\yy}u(t,x)\ge\theta_{1\yy}(x) ~ ~ {\rm uniformly ~ in ~ }[0,ct],\;\;\forall\;c\ge 0.
  \eess
On the other hand, by Proposition \ref{p2.4} again, $\lim_{t\to\yy}\max_{x\in[0,ct]}|\bar{u}(t,x)-\theta_{1\ep}(x)|=0$ for all $c\ge0$.
Then arguing as above, we can derive
 \[\limsup_{t\to\yy}u(t,x)\le \theta_{1\yy}(x) ~ ~ {\rm uniformly ~ in ~ }[0,ct],  ~ ~ \forall c\ge0,\]
 which, combined with the lower limit of $u$, completes the proof of (ii) in (2).

Conclusion (3) is parallel to (2), and we omit the proof.

(4) Let $(u_l,v_l)$ be defined in the proof of Theorem \ref{p2.6}. For the given $L>0$, as $\lim_{l\to\yy}(u_l,v_l)=(\tilde u,\tilde v)$ locally uniformly in $[0,\yy)$, there exists a large $l>L$ such that $(u_l,v_l)\ge(\tilde u-\ep,\tilde v-\ep)$ for $x\in[0,L]$, and $\lm_i^p(l)>0$, $i=1,2$. For the fixed $l$, due to $s_i(\yy)=\yy$, one can find a large $T_l>0$ such that $s_i(T_l)>l$, $i=1,2$. Then $(u,v)$ satisfies
 \bess
 \begin{cases}
 \dd u_t\ge d_1\int_{0}^{l}J_1(x-y)u(t,y)\dy-d_1u+f_1(u,v), &~ t>T_l, ~ x\in[0,l],\\[3mm]
 \dd v_t\ge d_2\int_{0}^{l}J_2(x-y)v(t,y)\dy-d_2v+f_2(u,v), & t>T_l, ~ x\in[0,l],\\
 u(T_l,x)>0, ~ v(T_l,x)>0, & x\in[0,l].
 \end{cases}
 \eess
 Let $(\ud u,\ud v)$ be the solution of
 \bess
 \begin{cases}
 \dd \ud u_t= d_1\int_{0}^{l}J_1(x-y)\ud u(t,y)\dy-d_1\ud u+f_1(\ud u,\ud v), & t>T_l, ~ x\in[0,l],\\[3mm]
 \dd \ud v_t=d_2\int_{0}^{l}J_2(x-y)\ud v(t,y)\dy-d_2\ud v+f_2(\ud u,\ud v), &t>T_l, ~ x\in[0,l],\\
 \ud u(T_l,x)=u(T_l,x), ~ \ud v(T_l,x)=v(T_l,x), & x\in[0,l].
 \end{cases}
 \eess
Owing to $\lm_i^p(l)>0$, $i=1,2$, by Proposition \ref{p2.5}, $\lim_{t\to\yy}(\ud u,\ud v)=(u_{l},v_{l})$ uniformly in $[0,l]$. There exists $T>T_l$ such that $(\ud u,\ud v)\ge(u_{l}-\ep,v_{l}-\ep)$  for $t\ge T$ and $x\in[0,l]$. Moreover, a comparison argument gives $(u, v)\ge(\ud u,\ud v)$ for $t\ge T$ and $x\in[0,l]$. Hence $(u,v)\ge(\tilde u-2\ep,\tilde v-2\ep)$ for $t\ge T$ and $x\in[0,L]$, which implies
  \bes
  \liminf_{t\to\yy}(u,v)\ge(\tilde u,\tilde v)\;\;\;\mbox{locally\, uniformly\, in }\; [0,\yy).
  \lbl{x3.4}\ees

On the other hand, let $(\bar{u},\bar{v})$ be the unique solution of
 \bess
 \begin{cases}
 \dd \bar u_t= d_1\int_{0}^{\yy}J_1(x-y)\bar u(t,y)\dy-d_1\bar u+f_1(\bar u,\bar v), ~ ~ t>0, ~ x\in[0,\yy),\\[3mm]
 \dd \bar v_t=d_2\int_{0}^{\yy}J_2(x-y)\bar v(t,y)\dy-d_2\bar v+f_2(\bar v,\bar v), ~ ~ t>0, ~ x\in[0,\yy),\\
 \bar u(0,x)=\|u_0\|_{\yy}+u^*, ~ \bar v(0,x)=\|v_0\|_{\yy}+v^*, ~ ~ x\in[0,\yy).
 \end{cases}
 \eess
Then $(u,v)\le(\bar{u},\bar{v})$ and $(\bar{u},\bar{v})$ is nonincreasing in $t\ge0$ by the comparison principle. It is easy to show that $\lim_{t\to\yy}(\bar{u},\bar{v})$ exists and is a bounded nonnegative solution of \eqref{2.11}. Besides, $(\bar{u},\bar{v})\ge(\ud u,\ud v)$, which implies that $\liminf_{t\to\yy}(\bar{u},\bar{v})\ge(\tilde u,\tilde v)$ locally uniformly in $[0,\yy)$. Thus, $\limsup_{t\to\yy}(u,v)\le\lim_{t\to\yy}(\bar{u},\bar{v})=(\tilde u,\tilde v)$ locally uniformly in $[0,\yy)$ by the uniqueness of bounded positive solution of \eqref{2.11}. This, combined with \qq{x3.4}, yields that $\lim_{t\to\yy}(u,v)=(\tilde u,\tilde v)$ locally uniformly in $[0,\yy)$.

Next we prove the estimates for spreading speeds of $s_i$ with $i=1,2$. We only handle $s_1$. Based on the proof of (2), by comparing $(u,s_1)$ and $(\ud u,\ud s_1)$, we have that if $J_1$ satisfies {\bf (J1)} then
$\liminf_{t\to\yy}{s_1(t)}/{t}\ge \ud c_1$, and if $J_1$ violates {\bf (J1)} then  $\lim_{t\to\yy}s_1(t)/t=\yy$. Let $(\bar{u},\bar{s}_1)$ be the unique solution of \eqref{2.4} with $(d,P,\alpha,\beta,\mu, h_0,w_0)=(d_1,J_1,r_1a,r_1,\mu_1,s_1(0),u_0)$. Then $(u(t,x),s_1(t))\le (\bar{u}(t,x),\bar{s}_1(t))$ for $t\ge0$ and $x\in[0,s_1(t)]$. By $s_1(\yy)=\yy$, we have $\bar{s}_1(\yy)=\yy$. It follows by Proposition \ref{p2.4} that
 \[\limsup_{t\to\yy}\frac{s_1(t)}{t}\le \lim_{t\to\yy}\frac{\bar{s}_1(t)}{t}=\bar c_1.\]
The proof is complete.
\end{proof}

Next we give some criteria determining when spreading or vanishing happens.

\begin{theorem}\label{t3.2} Let $(u,v,s_1,s_2)$ be the unique solution of \eqref{1.1}. Then the following statements are valid.
\begin{enumerate}[$(1)$]
  \item If $r_1a\ge d_1$ $(r_2\ge d_2)$, then spreading always happens for $(u,s_1)$ $((v,s_2))$ no matter what other parameters are.
  \item If $r_1a<d_1$ $(r_2<d_2)$, then there exists a unique $\ell_1>0$ $(\ell_2>0)$ such that spreading happens for $(u,s_1)$ $((v,s_2))$ when  $s_{10}\ge \ell_1$ $(s_{20}\ge\ell_2)$.
  \item If $r_1a<d_1$ and $s_{10}<\ell_1$ $(r_2<d_2 {\rm ~ and ~  } s_{20}<\ell_2)$, there exists a unique $\mu^*_1>0$ $(\mu^*_2>0)$ such that spreading happens for $(u,s_1)$ $((v,s_2))$ if and only if $\mu_1>\mu^*_1$ $(\mu_2>\mu^*_2)$.
\end{enumerate}
\end{theorem}

\begin{proof}We only prove the conclusions about $(u,s_1)$ since those of $(v,s_2)$ can be obtained by following similar lines.

(1) If $r_1a\ge d_1$, then $\lambda_1^p(l)>0$ for all $l>0$ by \qq{x2.2}, so spreading happens by Theorem \ref{t3.1}(1).

(2) If $r_1a<d_1$, by \qq{x2.2} again, there is a unique $\ell_1>0$ such that $\lambda_1^p(\ell_1)=0$ and $\lambda_1^p(l)(l-\ell_1)>0$ for $l\neq\ell_1$.  Due to  Theorem \ref{t3.1}(1), spreading happens when $s_{10}\ge\ell_1$.

(3) Owing to $r_1a<d_1$ and $s_{10}<\ell_1$, it follows from \cite[Theorem 3.3]{LLW22} that there is a small $\ud \mu_1>0$ and a large $\ol \mu_1>0$ such that vanishing happens for \eqref{2.4} with $(d,P,\alpha,\beta,\mu, h_0,w_0)=(d_1,J_1,r_1a,r_1,\mu_1,s_1(0),u_0)$ if $\mu_1<\ud \mu_1$, and spreading happens for \eqref{2.4} with $(d,P,\alpha,\beta,\mu, h_0,w_0)=(d_1,J_1,r_1a,2r_1,\mu_1,s_1(0),u_0)$ if $\mu_1>\ol \mu_1$. Moreover, it is easy to see that $(u,s_1)$ is a lower solution of \eqref{2.4} with $(d,P,\alpha,\beta,\mu, h_0,w_0)=(d_1,J_1,r_1a,r_1,\mu_1,s_1(0),u_0)$, and meanwhile is an upper solution of \eqref{2.4} with $(d,P,\alpha,\beta,\mu, h_0,w_0)=(d_1,J_1,r_1a,2r_1,\mu_1,s_1(0),u_0)$. Thus, if $\mu_1<\ud \mu_1$, then vanishing happens for $(u,s_1)$, while spreading occurs for $(u,s_1)$ if $\mu_1>\ol \mu_1$.

On the other hand, since \eqref{1.1} is a cooperative system, as in \cite{LL,CT} we can show $(u,s_1)$ is nondecreasing in $\mu_1$ by using a comparison principle for \eqref{1.1}. Then we can argue as in the proof of \cite[Theorem 3.9]{DL2010} to derive (3). The proof is complete.
\end{proof}

\section{More accurate estimates of solution component $(u,v)$ for spreading case}

From Theorem \ref{t3.1}(4) we can only assert that, when spreading happens,
\[\lim_{t\to\yy}(u(t,x),v(t,x))=(\tilde u(x),\tilde v(x))\;~ ~ {\rm  in  ~ }\; [C_{{\rm loc}}([0,\yy))]^2.\]
In this section, we will improve the above convergence result by using a general technical lemma and modifying the usual iteration method.

First, we consider a steady state problem on half space that will be used later. Let $k\in C([0,\yy))$ be nonincreasing, $k_{\yy}:=\lim_{x\to\yy}k(x)>0$, and $d, \alpha>0$ be constants. Consider the problem
\bes\label{4.1}
d\int_{0}^{\yy}P(x-y)U(y)\dy-dU+U(\alpha-k(x)U)=0, ~ ~ x\in[0,\yy).
\ees

\begin{proposition}\label{p4.1}Let $P$ satisfy {\bf (J)}. Then the following statements are valid.
\begin{enumerate}[$(1)$]
  \item \eqref{4.1} has a unique bounded positive solution $U_k\in C([0,\yy))$, and $U_k$ is nondecreasing for $x\ge0$ and $\lim_{x\to\yy}U_k(x)=\alpha/k_{\yy}$.
  \item $U_k$ is nonincreasing for $k$, i.e., $U_{k_1}(x)\ge U_{k_2}(x)$ when $k_1(x)\le k_2(x)$.
  \item Let $k_n\in C([0,\yy))$ be nonincreasing and $\lim_{x\to\yy}k_n(x)>0$. Denote by $U_{k_n}$ the unique bounded positive solution of \eqref{4.1} with $k(x)$ replaced by $k_n(x)$. Then if $k_n$ is nonincreasing for $n\ge1$, i.e., $k_n(x)\ge k_{n+1}(x)$ for $x\ge0$, and $k_n(x)\to k(x)$ in $L^{\yy}([0,\yy))$ as $n\to\yy$, then $U_{k_n}(x)\to U_{k}(x)$ in $L^{\yy}([0,\yy))$ as $n\to\yy$.
\end{enumerate}
 \end{proposition}

 \begin{proof}
(1) {\it Step 1: The existence}. We shall show the existence of $U_k$ by a monotone iteration method. In view of \cite[Lemma 2.7]{LLW22}, problem \eqref{4.1} with $k(x)$ replaced by $k(0)$ has a unique bounded positive solution $\ud U\in C([0,\yy))$ and $0<\ud U<\alpha/k(0)$. Let $\bar{U}=\alpha/k_{\yy}$ and $\omega=d+\alpha+1$. Define an operator $\Gamma$ by
 \[\Gamma(u)=\frac{1}{\omega}\left(d\int_{0}^{\yy}P(x-y)u(y)\dy-du+u(\alpha-k(x)u)+\omega u\right), ~ ~ u\in C([0,\yy))\cap L^{\yy}([0,\yy)).\]
It is easy to verify that $\Gamma(u)$ is nondecreasing for $u\in C([0,\yy))$ and $0\le u\le \bar{U}$. Moreover, one readily shows that $\Gamma(\ud U)\ge\ud U$ and $\Gamma(\bar{U})\le \bar{U}$. Then by an iteration argument, problem \eqref{4.1} has a positive solution $U$ with $\ud U\le U\le \bar{U}$. Notice that $\int_{0}^{\yy}P(x-y)U(y)\dy$ and $k(x)$ are continuous for $x\ge0$. From a quadratic formula, it follows that $U$ is also continuous for $x\ge0$.

{\it Step 2: The uniqueness}. To prove the uniqueness, we first show $\lim_{x\to\yy}U(x)=\alpha/k_{\yy}$. Clearly, $U_{{\rm inf}}:=\liminf_{x\to\yy}U(x)>0$. Suppose that $U_{{\rm inf}}<\alpha/k_{\yy}$. Then there exists a sequence $\{x_n\}$ with $x_n\nearrow\yy$ such that $U(x_n)\to U_{{\rm inf}}$ as $n\to\yy$. By \eqref{4.1}, we have
 \bess
 d\int_{0}^{\yy}P(x_n-y)U(y)\dy=dU(x_n)-U(x_n)(\alpha-k(x_n)U(x_n))\to dU_{{\rm inf}}-U_{{\rm inf}}(\alpha-k_{\yy}U_{{\rm inf}})<dU_{{\rm inf}}.
 \eess
 However, by Fatou's Lemma, we see
 \bess
 \liminf_{n\to\yy}\!\!\int_{0}^{\yy}P(x_n-y)U(y)\dy&=&\liminf_{n\to\yy}\!\!\int_{-x_n}^{\yy}P(y)U(x_n+y)\dy\\
 &=&\liminf_{n\to\yy}\!\!\int_{-\yy}^{\yy}P(y)U(x_n+y)\chi_{[x_n,\yy)}(y)\dy\\
 &\ge&\int_{-\yy}^{\yy}\liminf_{n\to\yy}P(y)U(x_n+y)\chi_{[x_n,\yy)}(y)\dy\\
 &\ge&\int_{-\yy}^{\yy}P(y)U_{{\rm inf}}\dy=U_{{\rm inf}},
 \eess
where $\chi_{[x_n,\yy)}$ is the characteristic function of $[x_n,\yy)$. So we derive a contradiction, and further $\liminf_{x\to\yy}U(x)\ge\alpha/k_{\yy}$.

We will prove $U_{{\rm sup}}:=\limsup_{x\to\yy}U(x)\le\alpha/k_{\yy}$. Assume on the contrary that $U_{{\rm sup}}>\alpha/k_{\yy}$. Let $x_n\nearrow\yy$ be such that $U(x_n)\to U_{{\rm sup}}$. Then, due to \eqref{4.1}, we have
 \bess
 d\int_{0}^{\yy}P(x_n-y)U(y)\dy=dU(x_n)-U(x_n)(\alpha-k(x_n)U(x_n))\to dU_{{\rm sup}}-U_{{\rm sup}}(\alpha-k_{\yy}U_{{\rm sup}})>dU_{{\rm sup}}.
 \eess
 On the other hand, using the dominated convergence theorem yields
 \bess
\lim_{n\to\yy} \int_{0}^{\yy}P(x_n-y)U(y)\dy&=&\lim_{n\to\yy}\!\!\int_{-x_n}^{\yy}P(y)U(x_n+y)\dy\\
&=&\lim_{n\to\yy}\!\!\int_{-\yy}^{\yy}P(y)U(x_n+y)\chi_{[x_n,\yy)}(y)\dy\\
&\le&\lim_{n\to\yy}\!\!\int_{-\yy}^{\yy}P(y)\sup_{x\ge x_n+y}U(x)\chi_{[x_n,\yy)}(y)\dy\\
&=&\int_{-\yy}^{\yy}P(y)U_{{\rm sup}}\dy=U_{{\rm sup}}.
 \eess
We get a contradiction. Thus $\lim_{x\to\yy}U(x)=\alpha/k_{\yy}$.

Now we show the uniqueness. Let $V$ be another bounded positive solution of \eqref{4.1}. As above, it can be deduced that $V\in C([0,\yy))$ and $\lim_{x\to\yy}V(x)=\alpha/k_{\yy}$. Then we can define
\[\ud\eta=\inf\{\eta>1: \eta U(x)\ge V(x), ~ ~ x\in[0,\yy)\}.\]
Clearly, $\ud\eta$ is well defined, and $\ud\eta U(x)\ge V(x)$ for $x\ge0$. We next show $\ud\eta=1$. Otherwise, if $\ud\eta>1$, then there must exists some $x_0\ge0$ such that $\ud\eta U(x_0)=V(x_0)$. By virtue of \eqref{4.1}, we have
\[d\int_{0}^{\yy}P(x_0-y)\ud\eta U(y)\dy-\ud\eta k(x_0)U^2(x_0)
=d\int_{0}^{\yy}P(x_0-y)V(y)\dy-k(x_0)V^2(x_0),\]
i.e.,
 \bess
 0\le d\int_{0}^{\yy}P(x_0-y)[\ud\eta U(y)-V(y)]\dy <k(x_0)[\ud\eta^2U^2(x_0)-V^2(x_0)]=0.\eess
This contradiction shows $\ud\eta=1$, and $U(x)\ge V(x)$ for $x\ge0$. Similarly, $V(x)\ge U(x)$ for $x\ge0$. The uniqueness is obtained.

{\it Step 3: The monotonicity}. We first prove that the unique bounded positive solution $U_K(x)$ of \eqref{4.1} is nondecreasing for $x\ge0$. For any $\sigma>0$, set $U^{\sigma}_k(x)=U_k(x+\sigma)$. Then thanks to \eqref{4.1}, we see
\[d\int_{-\sigma}^{\yy}P(x-y)U^{\sigma}_k(y)\dy-dU^{\sigma}_k(x)
+U^{\sigma}_k(x)(\alpha-k(x+\sigma)U^{\sigma}_k(x))=0, ~ ~ x\in[0,\yy).\]
Since $U$ is positive and $k(x)$ is nonincreasing, we have
\[d\int_{0}^{\yy}P(x-y)U^{\sigma}_k(y)\dy-dU^{\sigma}_k(x)
+U^{\sigma}_k(x)(\alpha-k(x)U^{\sigma}_k(x))\le0, ~ ~ x\in[0,\yy).\]
Then we can argue as in the Step 2 to show $U^{\sigma}_k(x)\ge U(x)$ for $x\ge0$. Due to the arbitrariness of $\sigma$, we derive that $U(x)$ is nondecreasing for $x\ge0$.

(2) Now we show $U_k$ is nonincreasing for $k$, i.e., $U_{k_1}(x)\ge U_{k_2}(x)$ for any $k_1(x)\le k_2(x)$.  Owing to \eqref{4.1}, it is easy to see that
   \[d\int_{0}^{\yy}P(x-y)U_{k_1}(y)\dy-dU_{k_1}+U_{k_1}(\alpha-k_2(x)U_{k_1})\le0, ~ ~ x\in[0,\yy).\]
Since $U_{k_1}(x)\to\alpha/k_1(\yy)$ and $U_{k_2}(x)\to\alpha/k_2(\yy)$ as $x\to\yy$ and $\alpha/k_1(\yy)\ge\alpha/k_2(\yy)$, where $k_i(\yy):=\lim_{x\to\yy}k_i(x)$, we can use the similar method as in the Step 2 to prove $U_{k_1}(x)\ge U_{k_2}(x)$ for $x\ge0$. The details are omitted here.

(3) From the properties of $k_n$ and statement (2) for $U_{k_n}$, it follows that $U_{k_n}\le U_{k_{n+1}}\le U_k$ for $n\ge1$. Therefore  $U(x):=\lim_{n\to\yy}U_{k_n}(x)$ is well-defined and positive for $x\ge0$. In view of the dominated convergence theorem, $U$ satisfies \eqref{4.1}. Then by statement (1), $U\equiv U_{k}$ in $[0,\yy)$. Using Dini's theorem, we get $U_{k_n}\to U_{k}$ locally uniformly in $[0,\yy)$ as $n\to\yy$.

Notice that $U_{k_n}\le U_{k}$. Hence it remains to show that for any small $\ep>0$, there exists a large $N$ such that for $n\ge N$, $U_{k}(x)-\ep\le U_{k_n}(x)$ for $x\ge0$. Thanks to $\lim_{n\to\yy}k_n=k$ in $L^{\yy}([0,\yy))$, we have $\lim_{n\to\yy}k_{n\yy}=k_{\yy}$, where $k_{n\yy}:=\lim_{x\to\yy}k_n(x)$. So for small $\ep>0$, there exists a large $N_1$ such that $k_{n\yy}\ge k_{\yy}-\ep/2$ for $n\ge N_1$. This, together with the monotonicity of $U_{k_n}$ on $n$, yields that $U_{k_n}(x)\ge U_{k}(x)-\ep$ for $n\ge N_1$ and $x\ge X>0$, where $X$ does not depend on $n$. Recalling $\lim_{n\to\yy}U_{k_n}(x)=U_{k}(x)$ locally uniformly in $[0,\yy)$, we immediately get the desired result. The proof is complete.
 \end{proof}

\begin{remark}\label{r4.1}By arguing as in the proof of \cite[Lemma 2.7]{LLW22}, \cite[Proposition 3.5]{CDLL} and \cite[Theorem 2.1]{BL}, we can show that there exists a large $L_0$ depending only on $d$, $\alpha$ and $P(x)$ such that for all $l>L_0$, problem
\bes\label{4.2}
\dd d\int_{0}^{L}P(x-y)\ud u_l(t,y)\dy-d\ud u_l+\ud u_l(\alpha-k(x)\ud u_l)=0, ~ ~ x\in[0,l]\ees
has a unique positive solution $u_l\in C([0,l])$ and $\lim_{l\to\yy}u_l(x)=U_k(x)$ locally uniformly in $[0,\yy)$.
\end{remark}

Let $s\in C([0,\yy))$ be strictly increasing, $s(0)>0$ and $s_{{\rm inf}}:=\liminf_{t\to\yy}s(t)/t>0$. We will study the asymptotical behaviors of the following nonlocal diffusion problem defined in a domain with a given moving boundary
\bes\label{4.3}
\begin{cases}
  u_t=\dd d\int_{0}^{s(t)}P(x-y)u(t,y)\dy-du+u(\alpha-k(x)u), & t>0,~ x\in[0,s(t)), \\
  u(t,s(t))=0, & t>0, \\
  u(0,x)>0,~ x\in[0,s(0)); ~ ~ u(0,s(0))=0, ~ u(0,x)\in C([0,s(0)]),
\end{cases}
\ees
where $d$, $\alpha$ and $k(x)$ are given in Proposition \ref{p4.1}.

\begin{proposition}\label{p4.2}Let $u$ be the unique solution of \eqref{4.3}. Then the following statements are valid.\vspace{-2mm}
\begin{enumerate}[$(1)$]
  \item $\lim_{t\to\yy}u(t,x)=U_k(x)$ locally uniformly in $[0,\yy)$.\vspace{-2mm}
  \item  If $P$ satisfies {\bf (J2)}, then $\lim_{t\to\yy}u(t,x)=U_{k}(x)$ uniformly in $[0,ct]$ for any $0\le c<\min\{s_{{\rm inf}}, C_*\}$, where $C_*$ is the minimal speed of \eqref{2.3} with $\beta=k_{\yy}$.\vspace{-2mm}
  \item If $P$ does not satisfy {\bf (J2)}, then $\lim_{t\to\yy}u(t,x)= U_{k}(x)$ uniformly in $[0,ct]$ for any $0\le c< s_{\rm inf}$.
\end{enumerate}\vspace{-2mm}
\end{proposition}

Due to its length, we will divide the proof into several parts.

\begin{proof}[Proof of Proposition {\rm\ref{p4.2}(1)}] We first prove $\liminf_{t\to\yy}u(t,x)\ge U_k(x)$ locally uniformly in $[0,\yy)$. For any $L>0$ and small $\delta>0$, according to Remark \ref{r4.1}, we can find a large $l>L_0+L$ such that problem \eqref{4.2} has a unique positive solution $u_l$, and $u_l(x)\ge U_k(x)-\delta$ for $x\in[0,L]$.
It is easy to check that for all $\ep\in(0,1)$, $u_l$ satisfies
 \bes\label{4.4}
 d\int_{0}^lP(x-y)\ep u_l(y)\dy-d\ep u_l+\ep u_l(\alpha-k(x)\ep u_l)\ge0, ~ ~ x\in[0,l].
 \ees
Let $\ud u$ be the unique solution of
\bess
\begin{cases}
 \ud u_t=\dd d\int_{0}^lP(x-y)\ud u(t,y)\dy-d\ud u+\ud u(\alpha-k(x)\ud u), & t>0,~ x\in[0,l], \\
  \ud u(0,x)=\ep u_l, & x\in[0,l].
\end{cases}
\eess
Clearly, $0<\ud u<\max\{\|\ep u_l\|_{\yy},\alpha/k_{\yy}\}$. Due to \eqref{4.4} and a comparison argument, we know that $\ud u(t,x)$ is nondecreasing in $t\ge0$. Thus, $\lim_{t\to\yy}\ud u(x,t)=u_l(x)$, and so $\liminf_{t\to\yy}u(t,x)\ge u_l(x)$ uniformly in $[0,l]$. As $u_l(x)\ge U_k(x)-\delta$ for $x\in[0,L]$, there exists $T>0$ such that $u(t,x)\ge U_k(x)-2\delta$ for $t\ge T$ and $x\in[0,L]$, which implies the desired conclusion.

It then remains to prove $\limsup_{t\to\yy}u(t,x)\le U_k(x)$ locally uniformly in $[0,\yy)$. Let $\bar{u}$ be the unique solution of
 \bess\begin{cases}
 \bar u_t=\dd d\int_{0}^{\yy}P(x-y)\bar u(t,y)\dy-d\bar u+\bar u(\alpha-k(x)\bar u), & t>0,~ x\in[0,\yy), \\
  \bar u(0,x)=\max\{\|u(0,x)\|_{\yy},\alpha/k_{\yy}\}, & x\in[0,\yy).
 \end{cases}\eess
By a comparison method (see \cite[Lemma 2.2]{LLW22}), $\bar{u}(t,x)\ge \ep U_k(x)$ for $t,x\ge 0$ and some $\ep>0$, and $\bar{u}(t,x)$ is nonincreasing in $t\ge0$. Thus  $\lim_{t\to\yy}\bar{u}(t,x)=U_k(x)$ locally uniformly in $[0,\yy)$. Besides, the  comparison principle gives $u(t,x)\le \bar{u}(t,x)$ for $t\ge0$ and $x\in[0,s(t)]$. So  the proof is ended.\end{proof}

In order to prove conclusions (2) and (3), we first prove
 \bes
 \limsup_{t\to\yy}u(t,x)\le U_k(x)\;\;\;\mbox{uniformly in}\;\;[0,\yy).
 \lbl{x4.4a}\ees
Clearly, it is sufficient to show that for any small $\delta>0$, there exists a large $T>0$ such that $u(t,x)\le U_k(x)+\delta$ for $t\ge T$ and $x\ge0$.

Using a comparison argument yields that $\limsup_{t\to\yy}u(t,x)\le \alpha/k_{\yy}$ uniformly in $[0,\yy)$. Thus there is a large $T_1>0$ such that $u(t,x)\le \alpha/k_{\yy}+\delta/2$ for $t\ge T_1$ and $x\ge0$. Notice that $U_k(x)\to\alpha/k_{\yy}$ as $x\to\yy$. We can find a large $X>0$ such that $U_k(x)+\delta>\alpha/k_{\yy}+\delta/2$. Hence $u(t,x)\le U_k(x)+\delta$ for $t\ge T_1$ and $x\ge X$. On the other hand, in light of (1), there is a large $T_2>0$ such that $u(t,x)\le U_k(x)+\delta$ for $t\ge T_2$ and $x\in[0,X]$. Then \eqref{x4.4a} is proved.

\begin{proof}[Proof of Proposition {\rm\ref{p4.2}(2)}] Notice that \qq{x4.4a}.  It suffices to prove $\liminf_{t\to\yy}u(t,x)\ge U_k(x)$ uniformly in $x\in[0,ct]$ for all $0\le c<\min\{s_{\rm inf},C_*\}$. Obviously, it is enough to show that for any small $\ep>0$, there exists a large $T>0$ such that
 \bes
 u(t,x)\ge U_k(x)-2\sqrt{\ep}\;\;\;\mbox{ for}\;\; t\ge T,\; x\in[0,ct].
  \lbl{x4.5}\ees
We will handle it in two situations.

{\it Case 1: $P$ has a compact support, i.e., ${\rm supp}\,P\in[-l_0,l_0]$ for some $l_0>0$}. Choose $c$ and $c_1$ satisfying $0<c<c_1<\min\{s_{{\rm inf}},C_*\}$. We can take $0<\ep\ll 1$ such that
  \[\frac{\alpha}{k_{\yy}}
-\frac{\alpha(1-\ep)}{k_{\yy}+\ep}\le \sqrt{\ep}.\]
For such $\ep$, using the properties of $k(x)$ and $U_k(x)$, we can choose $L\gg 2l_0$ such that $k(x)\le k_{\yy}+\ep$ and $U_k(x)\ge\frac{\alpha}{k_{\yy}+\ep}$ for $x\ge L$. By the result (1) and $c_1<s_{{\rm inf}}$, there is $T_1\gg 1$ so that
  \[s(t)\ge c_1t+2L,\;\;ct>L,\;\;u(t,x)\ge(1-\ep)U_k(x)
  \ge\frac{\alpha(1-\ep)}{k_{\yy}+\ep} ~ ~ {\rm for ~ }t\ge T_1, ~ x\in[L,2L].\]
Set $\ud s(t)=c_1(t-T_1)+2L$. In view of \eqref{4.3}, $u$ satisfies
 \bess\begin{cases}
  u_t\ge\dd d\int_{L}^{\ud s(t)}P(x-y)u(t,y)\dy-du+u[\alpha-(k_{\yy}+\ep)u], & t>T_1,~ x\in[2L,\ud s(t)], \\[2mm]
  u(t,\ud s(t))>0, \;\;
  u(t,x)\ge\dd\frac{\alpha(1-\ep)}{k_{\yy}+\ep}, ~ &t\ge T_1, ~ x\in[L,2L].
 \end{cases}\eess
Let $(c^{\mu}_{\ep},\phi)$ be the unique solution pair of \eqref{2.2} with $\beta=k_{\yy}+\ep$. By Proposition \ref{p2.3}, we know that $C_{*}$ is the minimal speed of \eqref{2.3} with $\beta$ replaced by $k_{\yy}+\ep$ and $\lim_{\mu\to\yy}c^{\mu}_{\ep}= C_{*}$. We can choose $\mu\gg 1$ such that $c_1<\min\{s_{{\rm inf}},c^{\mu}_{\ep}\}$. Define
\[\ud u(t,x)=(1-\ep)\phi(x-\ud s(t)) ~ ~ {\rm for ~ }t\ge T_1, ~ x\in[L,\ud s(t)].\]
We now show that $\ud u$ satisfies
\bes\label{4.5}
\begin{cases}
 \ud u_t\le\dd d\int_{L}^{\ud s(t)}P(x-y)\ud u(t,y)\dy-d\ud u+\ud u[\alpha-(k_{\yy}+\ep)\ud u], & t>T_1,~ x\in[2L,\ud s(t)], \\[2mm]
  \ud u(t,\ud s(t))=0, \;\;
 \ud u(t,x)\le u(t,x), ~ &t\ge T_1, ~ x\in[L,2L].
\end{cases}
\ees
Once we have done that, by a comparison method, we have
 \bes\label{4.6}
u(t,x)\ge \ud u(t,x) ~ ~ {\rm for ~ }t\ge T_1, ~ x\in[L,\ud s(t)].
\ees
Obviously, $\ud u(t,\ud s(t))=0$. Besides, we know that, for $t\ge T_1$ and $x\in[L,2L]$,
 \[u(t,x)\ge\frac{\alpha(1-\ep)}{k_{\yy}+\ep}\ge (1-\ep)\phi(x-\ud s(t))=\ud u(t,x).\]
So it remains to show the first inequality of \eqref{4.5}. Set $z=x-\ud s(t)$. Direct computations yield
\bess
\ud u_t&=&-c_1(1-\ep)\phi'(x-\ud s(t))\le -c^{\mu}_{\ep}(1-\ep)\phi'(x-\ud s(t))\\
&=&(1-\ep)\left(d\int_{-\yy}^{\ud s(t)}P(x-y)\phi(y-\ud s(t))\dy-d\phi(z)+\phi(z)[\alpha-(k_{\yy}+\ep)\phi(z)]\right)\\
&=&(1-\ep)\left(d\int_{L}^{\ud s(t)}P(x-y)\phi(y-\ud s(t))\dy-d\phi(z)+\phi(z)[\alpha-(k_{\yy}+\ep)\phi(z)]\right)\\
&\le&d\int_{L}^{\ud s(t)}P(x-y)\ud u(t,y)\dy-d\ud u+\ud u[\alpha-(k_{\yy}+\ep)\ud u].
\eess
Thus \eqref{4.5} holds, so does \qq{4.6}.

Now we prove \qq{x4.5}. Firstly, the direct calculations yields
\bess
\max_{x\in[L,ct]}|\ud u(t,x)-U_k(x)|&\le&\max_{x\in[L,ct]}\kk|\ud u(t,x)-\frac{\alpha(1-\ep)}{k_{\yy}+\ep}\rr|
+\max_{x\in[L,ct]}\kk|\frac{\alpha(1-\ep)}{k_{\yy}+\ep}-U_k(x)\rr|\\[1mm]
&&=(1-\ep)\left(\frac{\alpha(1-\ep)}{k_{\yy}+\ep}-\phi(ct-\ud s(t))\right)+\max_{x\in[L,ct]}\kk|\frac{\alpha(1-\ep)}{k_{\yy}+\ep}-U_k(x)\rr|\\[1mm]
&&\le(1-\ep)\left(\frac{\alpha(1-\ep)}{k_{\yy}+\ep}-\phi(ct-\ud s(t))\right)+\frac{\alpha}{k_{\yy}}-\frac{\alpha(1-\ep)}{k_{\yy}+\ep}\\[1mm]
&&\le(1-\ep)\left(\frac{\alpha(1-\ep)}{k_{\yy}+\ep}-\phi(ct-\ud s(t))\right)+\sqrt{\ep}\to\sqrt{\ep}
 \eess
as $t\to\yy$. There is $T_2\gg T_1$ such that $\max_{x\in[L,ct]}|\ud u(t,x)-U_k(x)|\le 2\sqrt{\ep}$, which, combined with \eqref{4.6}, implies that
\[u(t,x)\ge\ud u(t,x)\ge U_k(x)-2\sqrt{\ep} ~ ~ {\rm for ~ }t\ge T_2, ~ x\in[L,ct].\]
Secondly, owing to statement (1), there exists $T_3>0$ such that $u(t,x)\ge U_k(x)-2\sqrt{\ep}$ for $t\ge T_3$ and $x\in[0,L]$. Set $T=\max\{T_2,T_3\}$. Then \qq{x4.5} holds.

{\it Case 2: $P$ is not supported compactly}. Let
\begin{align*}
\xi(x)=\left\{\begin{aligned}
&1,& &|x|\le1,\\
&2-|x|,& &1\le|x|\le2,\\
&0, & &|x|\ge2
\end{aligned}\right. ~ ~ {\rm and ~ ~ }P_n(x)=\xi\kk(\frac{x}{n}\rr)P(x).
  \end{align*}
Clearly, $P_n$ are supported compactly for all $n\ge1$ and nondecreasing in $n$, $P_n\le P$, and $\lim_{t\to\yy}P_n=P$ in $L^1(\mathbb{R})\cap C_{\rm loc}(\mathbb{R})$. Take $n\gg 1$ such that $\alpha+d\|P_n\|_{L^1}-d>0$. Let $u_n$ be the unique solution of \eqref{4.3} with $P=P_n$. Then $u(t,x)\ge u_n(t,x)$ for $t\ge0$ and $x\in[0,s(t)]$, and $u_n$ satisfies
  \bess
\begin{cases}
  \partial_t u_n=\dd d_n\int_{0}^{s(t)}\hat P_n(x-y)u_n(t,y)\dy-d_nu_n+u_n(\alpha+d_n-d-k(x)u_n), & t>0,~ x\in[0,s(t)), \\
  u_n(t,s(t))=0, & t>0, \\
  u_n(0,x)=u(0,x),
\end{cases}
\eess
where $d_n=d\|P_n\|_{L^1}$ and $\hat{P}_n=P_n\|P_n\|^{-1}_{L^1}$.

For any small $\ep>0$, it follows from Proposition \ref{p2.2} that semi-wave problem
\begin{eqnarray*}\left\{\begin{array}{lll}
 d_n\displaystyle\int_{-\infty}^{0}\hat P_n(x-y)\phi(y){\rm d}y-d_n\phi+c\phi'+\phi[\alpha+d_n-d-(k_{\yy}+\ep)\phi]=0,\;\; x\in(-\infty,0),\\[3mm]
\phi(-\infty)=\dd\frac{\alpha+d_n-d}{k_{\yy}+\ep},\ \ \phi(0)=0, \ \ c=\mu\int_{-\infty}^{0}\int_0^{\infty}\!\hat P_n(x-y)\phi(x){\rm d}y{\rm d}x
 \end{array}\right.
 \end{eqnarray*}
has a unique solution pair $(c^{\mu,n}_{\ep},\phi_{n})$ with $c^{\mu,n}_{\ep}>0$ and $\phi'_n<0$. Moreover, by \cite[Theorem 4.4]{LLW22}, we know that if {\bf (J1)} holds, then $\lim_{n\to\yy}c^{\mu,n}_{\ep}=c^{\mu}_{\ep}$; if {\bf (J1)} does not hold, then $\lim_{n\to\yy}c^{\mu,n}_{\ep}=\yy$. For any given $0<c<c_1<\min\{s_{{\rm inf}},C_*\}$, we can choose $0<\ep\ll 1$ and $n\gg 1$ such that $c_1<c^{\mu,n}_{\ep}$. Then we can argue as in the step 1 to show that $\liminf_{t\to\yy}\ud u(t,x)\ge U^{n}_k(x)$ uniformly in $[0,ct]$, so
 \bes
 \liminf_{t\to\yy}u(t,x)\ge U^{n}_k(x)\;\;\;\mbox{ uniformly\, in}\;\;[0,ct],
 \lbl{x4.7}\ees
where $U^n_k\in C([0,\yy))$ is the unique bounded positive solution of
 \bess
d\int_{0}^{\yy}P_n(x-y)U(y)\dy-dU+U(\alpha-k(x)U)=0, ~ ~ x\in[0,\yy).
 \eess
And $U^n_k$ is strictly increasing to $(\alpha+d_n-d)/k_{\yy}$ as $x\to\yy$.

In the following we prove
 \bes
 \lim_{n\to\yy}U^n_k=U_k\;\;\;\mbox{uniformly\, in}\;\; [0,\yy).
 \lbl{x4.9}\ees
Once this is done, then by \qq{x4.7}, we can complete the proof of Step 2. In order to prove \qq{x4.9}, since $U^n_k\le U_k$, it is enough to show that for any given $0<\ep\ll 1$, there exists $N\gg 1$ such that $U^n_k(x)\ge U_k(x)-\ep$ for all $n\ge N$ and $x\ge0$.

As $P_n$ is nondecreasing in $n$ and $P_n\le P$, we can argue as in the proof of Proposition \ref{p4.1} to show that $U^n_k$ is also nondecreasing in $n$ and $U^n_k\le U_k$ for $x\ge0$. Thus $U^{\yy}(x):=\lim_{n\to\yy}U^n_k(x)$ is well defined, and $U^{\yy}\le U_{k}$. By the dominated convergence theorem, $U^{\yy}$ is a bounded positive solution of \eqref{4.1}. Then, by Proposition \ref{p4.1}, $U^{\yy}=U_k$. So  $\lim_{n\to\yy}U^n_k=U_k$ locally uniformly in $[0,\yy)$ by Dini's theorem.

On the other hand, it is clear that there exists $N_1$ such that
  \[\frac{\alpha+d_n-d}{k_{\yy}}>\frac{\alpha}{k_{\yy}}-\frac{\ep}{2}, ~ ~ \forall ~ n\ge N_1.\]
Since $U^{N_1}_k(x)\to(\alpha+d_{N_1}-d)/k_{\yy}$ and $U_k(x)\to \alpha/k_{\yy}$ as $x\to\yy$, one can find $X\gg 1$ such that $U^{N_1}_k(x)\ge U_k(x)-\ep$ for $x\ge X$. Recall that $U^n_k$ is also nondecreasing in $n$. Thus $U^{n}_k(x)\ge U_k(x)-\ep$ for $x\ge X$ and $n\ge N_1$.
Since we have shown that $\lim_{n\to\yy}U^n_k= U_k$ uniformly in $[0,X]$, there is $N_2$ such that $U^n_k(x)\ge U_k(x)-\ep$ for $n\ge N_2$ and $x\in[0,X]$. Thus, \qq{x4.9} holds.
\end{proof}

\begin{proof}[Proof of Proposition {\rm\ref{p4.2}(3)}] Remember that $P$ violates {\bf (J2)}. If {\bf (J1)} holds, then $\lim_{n\to\yy}c^{\mu,n}_{\ep}=c^{\mu}_{\ep}$ and $\lim_{\mu\to\yy}c^{\mu}_{\ep}=\yy$; if {\bf (J1)} does not hold, then $\lim_{n\to\yy}c^{\mu,n}_{\ep}\to\yy$. In either case, for any given $c\in[0, s_{\rm inf})$, we can choose $0<\ep\ll 1$ and $n,\mu\gg 1$ such that $c<c^{\mu,n}_{\ep}$. Then arguing as above, we can show that $\liminf_{t\to\yy}u(t,x)\ge U_k(x)$ uniformly in $[0,ct]$ for all $c\in[0, s_{\rm inf})$. This combined with \qq{x4.4a} finishes the proof.
\end{proof}

\begin{theorem}\label{t4.1}Let $(u,v,s_1,s_2)$ be the unique solution of \eqref{1.1}, and $(\tilde u,\tilde v)$ be the unique bounded positive solution of \eqref{2.11}. Let $\ud c_1$ and $ \ud c_2$ be given by \qq{x3.1}. Suppose that spreading occurs for \eqref{1.1}.
\begin{enumerate}[$(1)$]
  \item If $J_i$ with $i=1,2$ satisfy {\bf (J1)}, then
  \[\lim_{t\to\yy}(u(t,x),v(t,x))=(\tilde u(x),\tilde v(x)) ~ ~ {\rm uniformly ~ in ~ }[0,ct],\;\;\forall\;0\le c<\min\{\ud c_1,\,\ud c_2\}.\]
  \item If $J_1$  satisfies {\bf (J1)} and $J_2$ violates {\bf (J1)}, then
  \[\lim_{t\to\yy}(u(t,x),v(t,x))=(\tilde u(x),\tilde v(x)) ~ ~ {\rm uniformly ~ in ~ }[0,ct],\;\;\forall\;0\le c<\ud c_1.\]
 \item If $J_1$  does not satisfy {\bf (J1)} and $J_2$ satisfies {\bf (J1)}, then
  \[\lim_{t\to\yy}(u(t,x),v(t,x))=(\tilde u(x),\tilde v(x)) ~ ~ {\rm uniformly ~ in ~ }[0,ct],\;\;\forall\;0\le c<\ud c_2.\]
  \item  If both $J_i$ with $i=1,2$ violate {\bf (J1)}, then
  \[\lim_{t\to\yy}(u(t,x),v(t,x))=(\tilde u(x),\tilde v(x)) ~ ~ {\rm uniformly ~ in ~ }[0,ct],\;\;\forall\;c\ge0.\]
\end{enumerate}
\end{theorem}
\begin{proof}(1) {\it Step 1}. We first prove
 \bes
 \limsup_{t\to\yy}(u(t,x),v(t,x))\le (\tilde u(x),\tilde v(x))\;\;\;\mbox{uniformly \,in}\;\; [0,\yy).
 \lbl{x4.12}\ees
Clearly, it suffices to show that for any small $\ep>0$, there exists a $T>0$ such that
   \[(u(t,x),v(t,x))\le (\tilde u(x)+\ep,\tilde v+\ep) ~ ~ {\rm for ~ }t\ge T, ~ x\in[0,\yy).\]
Let $(\bar{u}(t),\bar{v}(t))$ be the unique solution of the following ODE system
\bess
\begin{cases}
\bar{u}_t=f_1(\bar{u},\bar{v}), ~\; \bar{v}_t=f_2(\bar{u},\bar{v}), ~ \;t>0,\\
\bar{u}(0)=K_1, ~ \bar{v}(0)=K_2,
\end{cases}
\eess
where $(K_1,K_2)$ is large enough such that $(K_1,K_2)\gg(\max\{\|u_0\|_{\yy},u^*\},\max\{\|v_0\|_{\yy},v^*\})$. A phase plane analysis arrives at $\lim_{t\to\yy}(\bar{u}(t),\bar{v}(t))=(u^*,v^*)$.
By a comparison argument, we have $\limsup_{t\to\yy}(u(t,x),v(t,x))\le (u^*,v^*)$ uniformly in  $x\in[0,\yy)$. For any small $\ep>0$, there exists some $T_1>0$ such that $(u(t,x),v(t,x))\le (u^*+\frac{\ep}{2},v^*+\frac{\ep}{2})$ for $t\ge T_1$ and $x\ge0$. On the other hand, since $\lim_{x\to\yy}(\tilde u(x),\tilde v(x))=(u^*,v^*)$, there exists a large $X>0$ such that $(\tilde u(x)+\ep,\tilde v(x)+\ep)\ge(u^*+\frac{\ep}{2},v^*+\frac{\ep}{2})$ for $x\ge X$. It then follows that
\[(u(t,x),v(t,x))\le (u^*+\ep/2,v^*+\ep/2)\le(\tilde u(x)+\ep,\tilde v(x)+\ep) ~ ~ {\rm for ~ }t\ge T_1, ~x\in[X,\yy).\]

For $x\in[0,X]$, by Theorem \ref{t3.1}(4), we can find a $T_2>0$ such that
\[(u(t,x),v(t,x))\le (\tilde u(x)+\ep,\tilde v(x)+\ep) ~ ~ {\rm for ~ }t\ge T_2, ~ x\in[0,X].\]
Setting $T=\max\{T_1,T_2\}$, we get \qq{x4.12}.

{\it Step 2}. We will show that for any given $0\le c<\min\{\ud c_1,\,\ud c_2\}$, there holds:
 \bes
 \liminf_{t\to\yy}(u(t,x),v(t,x))\ge(\tilde u(x),\tilde v(x)) ~ ~ {\rm uniformly ~ in ~ }[0,ct].\lbl{x4.13}\ees
To this aim, we first construct a sequence $\{(\ud U_n,\ud V_n)\}$ satisfying $\ud U_1=\theta_{1\yy}$ and
 \bes
\begin{cases}
 \dd d_2\int_{0}^{\yy}J_2(x-y)\ud V_n(y)\dy-d_2\ud V_n+f_2(\ud U_n,\ud V_n)=0, ~ ~ x\in[0,\yy),\\[3mm]
\dd d_1\int_{0}^{\yy}J_1(x-y)\ud U_{n+1}(y)\dy-d_1\ud U_{n+1}+f_1(\ud U_{n+1},\ud V_n)=0, ~ ~ x\in[0,\yy),
 \end{cases}\lbl{x4.14}\ees
and $\ud U_n,\ud V_n\in C([0,\yy))$, such that
 \bes
 &&\liminf_{t\to\yy}(u(t,x),v(t,x))\ge (\ud U_n(x), \ud V_n(x)) ~ ~ {\rm uniformly ~ in ~ }[0,ct],\label{x4.15}\\[2mm]
 &&\lim_{n\to\yy}(\ud U_n(x),\ud V_n(x))=(\tilde u(x),\tilde v(x)) ~ ~ {\rm locally \, uniformly\, in}\;\;[0,\yy).\lbl{x4.16}\ees

Choose $c<c_n<\min\{\ud c_1,\,\ud c_2\}$ satisfying $c_n\searrow c$ as $n\to\yy$. Note that $u$ satisfies
\bess
\begin{cases}
u_t \ge\dd d_1\int_{0}^{s_1(t)}J_1(x-y)u(t,y)\dy-d_1u+r_1u(a-2u), &t>0, ~ x\in[0,s_1(t)),\\
u(t,s_1(t))=0, &t>0,\\
u(0,x)=u_0(x).
\end{cases}
\eess
Making use of Proposition \ref{p4.2} and a comparison method, we have  $\liminf_{t\to\yy}u(t,x)\ge \ud U_1$ uniformly in $x\in[0,c_1t]$, where $\ud U_1\in C([0,\yy))$ is the unique bounded positive solution of
 \[d_1\int_{0}^{\yy}J_1(x-y)\ud U_1(y)\dy-d_1\ud U_1+r_1\ud U_1(a-2\ud U_1)=0, ~ ~ x\in[0,\yy).\]
For the given $0<\ep\ll 1$, there exists $T_\ep\gg 1$ such that $s_2(t)>c_2t+s_{20}$ and $u(t,x)\ge U_1(x)-\ep$ for $t\ge T_\ep$ and $x\in[0,c_2t+s_{20}]$. Thus $v$ satisfies
\bess
\begin{cases}
v_t\ge \dd d_2\int_{0}^{c_2t+s_{20}}J_2(x-y)v(t,y)\dy-d_2v+f_2(\ud U_1-\ep,v), &t>T_\ep, ~ x\in[0,c_2t+s_{20}],\\
v(t,c_2t+s_{20})>0, &t>T_\ep,\\
v(T_\ep,x)>0, &x\in[0,c_2 T_\ep+s_{20}].
\end{cases}
\eess
Making use of a comparison consideration and Proposition \ref{p4.2}, we have
$\liminf_{t\to\yy}v(t,x)\ge \ud V^{\ep}_1(x)$ uniformly in $x\in[0,c_3t]$,
where $\ud V^{\ep}_1\in C([0,\yy))$ is the unique bounded positive solution of
  \[d_2\int_{0}^{\yy}J_2(x-y)\ud V^{\ep}_1(y)\dy-d_2\ud V^{\ep}_1+f_2(\ud U_1-\ep,\ud V^{\ep}_1)=0, ~ ~ x\in[0,\yy).\]
Due to Proposition \ref{p4.1}, we have $\liminf_{t\to\yy}v(t,x)\ge \ud V_1(x)$ uniformly in $x\in[0,c_3t]$, where $\ud V_1\in C([0,\yy))$ is the unique bounded positive solution of
 \[d_2\int_{0}^{\yy}J_2(x-y)\ud V_1(y)\dy-d_2\ud V_1+f_2(\ud U_1,\ud V_1)=0, ~ ~ x\in[0,\yy).\]

For the given $0<\ep\ll 1$, there exists $T_\ep\gg 1$ such that $s_1(t)>c_4t+s_{10}$ and $v(t,x)\ge \ud V_1(x)-\ep$ for $t\ge T_\ep$ and $x\in[0,c_4t+s_{10}]$. Then $u$ satisfies
\bess
\begin{cases}
u_t\ge \dd d_1\int_{0}^{c_4t+s_{10}}J_1(x-y)u(t,y)\dy-d_1u+f_1(u,\ud V_1-\ep), &t>T_\ep, ~ x\in[0,c_4t+s_{10}],\\
u(t,c_4t+s_{10})>0, &t>T_\ep,\\
u(T_\ep,x)>0, &x\in[0,c_4T_\ep+s_{10}].
\end{cases}
\eess
Again, by a comparison argument and Proposition \ref{p4.2}, $\liminf_{t\to\yy}u(t,x)\ge \ud U^{\ep}_2(x)$ uniformly in $x\in[0,c_5t]$,
where $\ud U^{\ep}_2\in C([0,\yy))$ is the unique bounded positive solution of
\[d_1\int_{0}^{\yy}J_1(x-y)\ud U^{\ep}_2(y)\dy-d_1\ud U^{\ep}_2+f_1(\ud U_2,\ud V_1-\ep)=0, ~ ~x\in[0,\yy).\]
Again owing to Proposition \ref{p4.1}, we have $\liminf_{t\to\yy}u(t,x)\ge \ud U_2(x)$ uniformly in $x\in[0,c_5t]$, where $\ud U_2\in C([0,\yy))$ is the unique bounded positive solution of
   \[d_1\int_{0}^{\yy}J_1(x-y)\ud U_2(y)\dy-d_1\ud U_2+f_1(\ud U_2,\ud V_1)=0, ~ ~ x\in[0,\yy).\]
Similarly, we can derive $\liminf_{t\to\yy}v(t,x)\ge \ud V_2(x)$ uniformly in  $x\in[0,c_7t]$, where $\ud V_2\in C([0,\yy))$ is the unique bounded positive solution of
  \[d_2\int_{0}^{\yy}J_2(x-y)\ud V_2(y)\dy-d_2\ud V_2+f_2(\ud U_2,\ud V_2)=0, ~ ~  x\in[0,\yy).\]

Repeating the above process, we can get a sequence $\{(\ud U_n,\ud V_n)\}$ satisfying \qq{x4.14} and \qq{x4.15}, and $(\ud U_n,\ud V_n)\in [C([0,\yy))]^2$. Moreover, by virtue of Proposition \ref{p4.1}, we can prove that
\[\ud U_1(x)\le \ud U_n(x)\le \ud U_{n+1}(x)\le u^*, ~ ~ ~ \ud V_1(x)\le \ud V_n(x)\le \ud V_{n+1}(x)\le v^*, ~ ~x\in[0,\yy).\]
By the dominated convergence theorem, Theorem \ref{p2.6} and Dini's theorem, we see that \qq{x4.16} holds.

Recall \qq{x4.15}. To prove \qq{x4.13}, it suffices to show that $\lim_{n\to\yy}(\ud U_n(x),\ud V_n(x))=(\tilde u(x),\tilde v(x))$ uniformly in $[0,\yy)$. For clarity, we set $(\ud U_n(\yy),\ud V_n(\yy)):=\lim_{x\to\yy}(\ud U_n(x),\ud V_n(x))$. By the definition of $(\ud U_n,\ud V_n)$ and Proposition \ref{p4.1}, we have
\bess
\ud U_1(\yy)=\frac{a}{2}, ~1-\ud V_n(\yy)-\frac{\ud V_n(\yy)}{1+q\ud U_n(\yy)}=0, ~ a-\ud U_{n+1}(\yy)-\frac{\ud U_{n+1}(\yy)}{1+b\ud V_n(\yy)}=0.
\eess
Then it is easy to derive that $\lim_{n\to\yy}(\ud U_n(\yy),\ud V_n(\yy))\to(u^*,v^*)$. Then for the given $0<\ep\ll 1$, we can find $N_1\gg 1$ such that $(\ud U_{N_1}(\yy),\ud V_{N_1}(\yy))\ge(u^*-\frac{\ep}{2},v^*-\frac{\ep}{2})$. Noticing that $\lim_{x\to\yy}(\tilde u(x),\tilde v(x))=(u^*,v^*)$. By the continuity, there exists $X\gg 1$ such that $(\ud U_{N_1}(x),\ud V_{N_1}(x))\ge (\tilde u(x)-\ep,\tilde v(x)-\ep)$ for $x\ge X$. Since $(\ud U_n,\ud V_n)$ is nondecreasing in $n\ge1$, we conclude
 \[(\ud U_n(x),\ud V_n(x))\ge (\tilde u(x)-\ep,\tilde v(x)-\ep) ~ ~ {\rm for ~ }n\ge N_1, ~ x\in[X,\yy).\]
On the other hand, by \qq{x4.16}, there exists $N_2\gg 1$ such that
  \[(\ud U_n(x),\ud V_n(x))\ge (\tilde u(x)-\ep,\tilde v(x)-\ep) ~ ~ {\rm for ~ }n\ge N_2, ~ x\in[0, X].\]
Taking $N=\max\{N_1,N_2\}$, we obtain that
  \[(\ud U_n(x),\ud V_n(x))\ge (\tilde u(x)-\ep,\tilde v(x)-\ep) ~ ~ {\rm for ~ }n\ge N, ~ x\in[0,\yy),\]
which, combined with $(\ud U_n(x),\ud V_n(x))\le (\tilde u(x),\tilde v(x))$ for $x\ge0$, proves the desired result.

(2) Since $J_2$ does not satisfy {\bf (J1)}. According to Theorem \ref{t3.1}, we have $\lim_{t\to\yy}s_2(t)/t=\yy$. Thus similar to the arguments in the proof of statement (1), we can derive that
\[\liminf_{t\to\yy}(u(t,x),v(t,x))\ge (\tilde u(x), \tilde v(x)) ~ ~ {\rm uniformly ~ in ~ }[0,ct]\]
for any $0\le c<\ud c_1$. Note that \eqref{x4.15} holds in the present situation. The statement (2) is proved.

(3)-(4)  Using Theorem \ref{t3.1} and the analogous lines as above, one can prove statements (3) and (4). The details are omitted here.
The proof is ended.
\end{proof}

\section{Discussion--biological significance of the conclusions}
\setcounter{equation}{0}

In this paper, we investigated a free boundary problem which describes the interactions of two mutually beneficial species in a one-dimensional habitat. In this model, one species occupies the spatial domain $[0,s_1(t)]$, while the other is distributed over $[0,s_2(t)]$. The two free boundaries $x=s_1(t)$ and $x=s_2(t)$,  describing the spreading fronts of two mutually beneficial species, respectively, may intersect each other as time evolves. Moreover, it is easy to verify that the nonlinear term $(f_1(u,v),f_2(u,v))$ is reducible at $(0,0)$. Hence it seems very difficult to understand the whole dynamics of this model, especially for the following problems.

{\bf (A)}\, When $u$ and $v$ spread successfully, the precise spreading speeds of $s_i$ for $i=1,2$ are left open in this paper. The difficulty comes from a lack of understanding of the corresponding semi-wave problem since \eqref{1.1} is reducible. But we conjecture that $u$ and $v$ will spreads at different speeds, and even one speed is finite and the other is infinite.

{\bf (B)}\, In Theorem \ref{t3.1}, we proved that if {\bf (J1)} holds for $J_i$ with $i=1,2$ and spreading happens for both $u$ and $v$, then the spreading speed of their spreading fronts $s_1(t)$ and $s_2(t)$ have different upper and lower bounds:
 \bess
&\ud c_1\le\dd\liminf_{t\to\yy}\frac{s_1(t)}{t}
\le\limsup_{t\to\yy}\frac{s_1(t)}{t}\le \bar{c}_1,\\
 & \ud c_2\le\dd\liminf_{t\to\yy}\frac{s_2(t)}{t}\le \limsup_{t\to\yy}\frac{s_2(t)}{t}\le \bar{c}_2.
  \eess

For each $c\in(0,\,\min\{\ud c_1,\ud c_2\})$, in the spatial domain $[0,ct]$ two species $u$ and $v$ help each other and survive together, and $(u,v)$ converges to a positive steady state $(\tilde u,\tilde v)$ uniformly in $[0,ct]$ as $t\to\yy$ (see Theorem \ref{t4.1}). However, for the case $\bar c_1<\ud c_2$, which can be guaranteed by letting $\mu_1$ small enough and $\mu_2$ sufficiently large, then for all large $t$, only the species $v$ survives in the spatial domain $[c_1t,c_2t]$ with $\bar c_1<c_1<c_2<\ud c_2$, i.e., the species $v$ will not receive the help of $u$. We guess that $v$ will converge to the relatively small positive steady state $\theta_{2\yy}$ uniformly in $[c_1t,c_2t]$ as $t\to\yy$. Similar situations will appear if $J_1$ satisfies {\bf (J1)}, but $J_2$ violates {\bf (J1)}.

We would like mention that if the above conjectures hold true, then the propagation phenomenon here will be analogous to those of the L-V competition model with different free boundaries, which has been systematically studied by the authors of \cite{DW2018,Wu2019,KSTD,DW2022}.


\begin{thebibliography}{99}
\bibliographystyle{siam}
\setlength{\baselineskip}{15pt}


\vspace{-1.5mm}\bibitem{CDLL}J.-F. Cao, Y.H. Du, F. Li and W.-T. Li, {\it The dynamics of a Fisher-KPP nonlocal diffusion model with free boundaries}, J. Funct. Anal., \textbf{277} (2019), 2772-2814.

\vspace{-1.5mm}\bibitem{DLZ}Y.H. Du, F. Li and M.L. Zhou, {\it Semi-wave and spreading speed of the nonlocal Fisher-KPP equation with free boundaries}, J. Math. Pures Appl.,  \textbf{154} (2021), 30-66.

\vspace{-1.5mm}\bibitem{DL2010}Y.H. Du and Z.G. Lin, {\it Spreading-vanishing dichotomy in the diffusive logistic model with a free boundary}, SIAM J. Math. Anal., \textbf{42} (2010), 377-405.

\vspace{-1.5mm}\bibitem{DN1}Y.H. Du and W.J. Ni, {\it Rate of propagation for the Fisher-KPP equation with nonlocal diffusion and free boundaries}, J. Eur. Math. Soc., (2023), Doi:10.4171/JEMS/1392.

\vspace{-1.5mm}\bibitem{DN2}Y.H. Du and W.J. Ni, {\it Exact rate of accelerated propagation in the Fisher-KPP equation with nonlocal diffusion and free boundaries}, Math. Ann., \textbf{389} (2024), 2931-2958.

\vspace{-1.5mm}\bibitem{DN3}Y.H. Du and W.J. Ni, {\it The high dimensional Fisher-KPP nonlocal diffusion equation with free boundary and radial symmetry, Part 2},  J. Funct. Anal., \textbf{287} (2024), Article No. 110649.


\vspace{-1.5mm}\bibitem{DN20}Y.H. Du and W.J. Ni, {\it Analysis of a West Nile virus model with nonlocal diffusion and free boundaries}, Nonlinearty, \textbf{33} (2020), 4407-4448.

\vspace{-1.5mm}\bibitem{ZLW20}M. Zhao, W.T. Li and Y.H. Du, {\it The effect of nonlocal reaction in an epidemic model with nonlocal diffusion and free boundaries}, Commun. Pure Appl. Anal. {\bf 19} (2020), 4599-4620.

\vspace{-1.5mm}\bibitem{ZZLW20} M. Zhao, Y. Zhang, W.-T. Li and Y.H. Du, {\it The dynamics of a degenerate epidemic model with nonlocal diffusion and free
boundaries}, J. Differ. Equ., {\bf 267} (2020), 3347-3386.

\vspace{-1.5mm}\bibitem{LWWcpaa}L. Li, J.P. Wang and M.X. Wang, {\it The dynamics of nonlocal diffusion systems with different free boundaries}, Comm. Pure Appl. Anal., \textbf{19} (2020), 3651-3672.

\vspace{-1.5mm}\bibitem{DuNi22}Y.H. Du and W.J. Ni, {\it Spreading speed for monostable cooperative systems with nonlocal diffusion and free boundaries,
part 1: semi-wave and a threshold condition}, J. Differ. Equ., {\bf 308} (2022),  369-420.

\vspace{-1.5mm}\bibitem{CDu22}T.-Y. Chang and Y.H. Du, {\it Long-time dynamics of an epidemic model with nonlocal diffusion and free boundaries}, Electron. Res. Arch., {\bf 30}(1) (2022), 289-313.

\vspace{-1.5mm}\bibitem{LLW22}L. Li, W.-T. Li and M.X. Wang, {\it Dynamics for nonlocal diffusion problems with a free boundary}, J. Differ. Equ., \textbf{330}  (2022), 110-149.

\vspace{-1.5mm}\bibitem{NV} T.-H. Nguyen and H.-H. Vo, {\it Dynamics for a two-phase free boundaries system in an epidemiological model with couple nonlocal dispersals}, J. Differ. Equ., \textbf{335} (2022), 398-463.

\vspace{-1.5mm}\bibitem{WDu22}R. Wang and Y.H. Du, {\it Long-time dynamics of a nonlocal epidemic model with free boundaries: Spreading-vanishing dichotomy}, J. Differ. Equ.,  \textbf{327}  (2022), 322-381.

\vspace{-1.5mm}\bibitem{DN200}Y.H. Du and W.J. Ni, {\it Approximation of random diffusion equation by nonlocal diffusion equation in free boundary problems of one space dimension}, Comm. Contemp. Math., \textbf{25} (2023), Article No. 2250004.

\vspace{-1.5mm}\bibitem{PLL23}L.Q. Pu, Z.G. Lin and Y. Lou, {\it A west nile virus nonlocal model with free boundaries and seasonal succession}, J. Math. Biol., \textbf{86}(2) (2023), Article No. 25.

\vspace{-1.5mm}\bibitem{ZWcnsns23}Q.Y. Zhang and M.X. Wang, {\it A nonlocal diffusion competition model with seasonal succession and free boundaries}, Commun. Nonlinear Sci. Numer. Simulat., \textbf{122} (2023), Article No. 107263

\vspace{-1.5mm}\bibitem{LW24jdde} L. Li and M.X. Wang, {\it Free boundary problems of a mutualist model with nonlocal diffusion}, J. Dyn. Diff. Equat., \textbf{36} (2024), 375-403.

\vspace{-1.5mm}\bibitem{LW24}L. Li and M.X. Wang, {\it Spreading dynamics of a Fisher-KPP nonlocal diffusion model with a free boundary }, submitted,  (2024), 	 arXiv: 2409.16101.

\vspace{-1.5mm}\bibitem{Du22}Y.H. Du, {\it Propagation and reaction-diffusion models with free boundaries}, Bull. Math. Sci. \textbf{12} (2022), Article No. 2230001.


\vspace{-1.5mm}\bibitem{AMRT}F. Andreu, J.M. Maz{\'o}n, J.D. Rossi and J. Toledo, {\it Nonlocal Diffusion Problems}, Math. Surveys Monogr. \textbf{165}, AMS, Providence, RI 2010.



\vspace{-1.5mm}\bibitem{LL}M. Li and Z.G. Lin, {\it The spreading fronts in a mutualistic model with advection}, Discrete Contin. Dyn. Syst. B, \textbf{20} (2015), 2089-2105.

\vspace{-1.5mm}\bibitem{ZW}Q.Y. Zhang and M.X. Wang, {\it Dynamics for the diffusive mutualist model with advection and different free boundaries}, J. Math. Anal. Appl., \textbf{474} (2019), 1512-1535.

\vspace{-1.5mm}\bibitem{CT}Q.L. Chen and Z.D. Teng, {\it Dynamics for a reaction-diffusion-advection mutualist model with different free boundaries}, Math. Meth. Appl. Sci., \textbf{46} (2023), 4965-4984.

\vspace{-1.5mm}\bibitem{CC}J. Carr and A. Chmaj, {\it Uniqueness of travelling waves for nonlocal monostable equations}, Proc. Amer. Math. Soc., \textbf{132} (2004), 2433-2439.

\vspace{-1.5mm}\bibitem{Ya}H. Yagisita, {\it Existence and nonexistence of traveling waves for a nonlocal monostable equation}, Publ. Res. Inst. Math. Sci., \textbf{45} (2009), 925-953.


\vspace{-1.5mm}\bibitem{BL}X.L. Bai and F. Li, {\it Classification of global dynamics of competition models with nonlocal dispersals I: symmetric kernels}, Calc. Var. Partial Differ. Equ., \textbf{57} (2018), Article No. 144.

\vspace{-1.5mm}\bibitem{DW2018}Y.H. Du and C.-H. Wu, {\it Spreading with two speeds and mass segregation in a diffusive competition system with free boundaries},  Calc. Var. Partial Differ. Equ., \textbf{57} (2018), Article No. 52.

\vspace{-1.5mm}\bibitem{Wu2019} C.-H. Wu, {\it Different spreading speeds in a weak competition model with two free boundaries}, J. Differ. Equ., \textbf{267} (2019), 4841-4862.

\vspace{-1.5mm}\bibitem{KSTD} K. Khan, S. Liu, T.M. Schaerf and Y.H. Du, {\it Invasive behaviour under competition via a free boundary model: A numerical approach}, J. Math. Biol., \textbf{83} (2021), Article No. 23.

\vspace{-1.5mm}\bibitem{DW2022}Y.H. Du and C.-H. Wu, {\it Classification of the spreading behaviors of a two-species diffusion-competition system with free boundaries},  Calc. Var. Partial Differ. Equ., \textbf{61} (2022), Article No. 54.
\end{thebibliography}
\end{document}